\documentclass[11pt]{amsart}

\usepackage{amsmath,amssymb,latexsym, dsfont}
\usepackage[small]{caption}
\usepackage{graphicx,color,mathrsfs,tikz}
\usepackage{subfigure,color,bm}
\usepackage[colorlinks=true,urlcolor=blue,
citecolor=red,linkcolor=blue,linktocpage,pdfpagelabels,
bookmarksnumbered,bookmarksopen]{hyperref}
\usepackage[italian,english]{babel}
\usepackage{units}
\usepackage{enumitem}
\usepackage[left=2.5cm,right=2.5cm,top=2.9cm,bottom=2.9cm]{geometry}
\usepackage[hyperpageref]{backref}
\usepackage{fancyhdr, dsfont}
\usepackage{tikz}
\usepackage{hyperref,cleveref}

\def\XXint#1#2#3{{\setbox0=\hbox{$#1{#2#3}{\int}$ }
		\vcenter{\hbox{$#2#3$ }}\kern-.6\wd0}}

\usepackage{amscd}
\usepackage{esint}

\newcommand{\eps}{\varepsilon}

\theoremstyle{plain}

\newtheorem{thm}{Theorem}[section]
\newtheorem*{thm*}{Theorem}
\theoremstyle{plain}
\newtheorem{lem}[thm]{Lemma}
\newtheorem{pro}{Proposition}[section]

\theoremstyle{definition}
\newtheorem{defi}{Definition}[section]

\newtheorem{remark}{Remark}[section]

\newcommand{\mR}{\mathbb{R}}

\newcommand{\Xr}{X^R}
\newcommand{\Xl}{X^L}

\newcommand{\ro}{\mathbb{R}}

\newcommand{\wsp}{\dot W^{s,p}}

\newcommand{\mcal}[1]{\mathcal{#1}}
\newcommand{\md}[1]{\left|#1 \right|}
\newcommand{\nrm}[1]{\left|\left| #1 \right|\right|}
\newcommand{\bct}[1]{\left(#1\right)}

\newcommand{\Sb}[1]{\left \lbrace #1\right\rbrace}

\newcommand{\dsp}{\displaystyle}

\newcommand{\cT}{{\mathcal T}}

\newcommand{\R}{\mathbb{R}}
\newcommand{\mZ}{\mathbb{Z}}

\numberwithin{equation}{section} \allowdisplaybreaks

        \title[Interpolation inequalities associated with Sobolev-Coulomb spaces]{Gagliardo-Nirenberg and  Caffarelli-Kohn-Nirenberg interpolation inequalities associated with Sobolev-Coulomb spaces}

        \author[A. Mallick]{Arka Mallick}
        \address[A. Mallick]{Department of Mathematics \newline\indent
	Technion,Haifa,Israel}
\email{arkamallick02@gmail.com, arka.mallick@campus.technion.ac.il}

        \author[H.-M. Nguyen]{Hoai-Minh Nguyen}
  \address[H.-M. Nguyen]{Laboratoire Jacques Louis Lions,  \newline\indent
	Sorbonne Universit\'e \newline\indent
	4 Place Jussieu, 75252, Paris, France}
\email{hoai-minh.nguyen@sorbonne-universite.fr}

\begin{document}

\begin{abstract}  
We establish the full range Gagliardo-Nirenberg and the Caffarelli-Kohn-Nirenberg interpolation inequalities associated with Sobolev-Coulomb spaces for the (fractional) derivative $0 \le  s \le 1$.   As a result, we rediscover known Gaglairdo-Nirenberg interpolation type inequalities associated  Sobolev-Coulomb spaces which were previously established in the scale of  $H^{s}$ with $0 < s \le 1$ and extend them for the full range $W^{s, p}$ with $0\le s \le 1$ and $1 < p < + \infty$.  Using these newly established weighted inequalities, we derive a new family of one body  Hardy-Lieb-Thirring inequalities and  use it to establish   a new  family of many body Hardy-Lieb-Thirring inequalities with a strong repulsive interaction term in $L^p$ scale.
\end{abstract}

\maketitle


\tableofcontents

\section{Introduction}

This paper is devoted to the Gagliardo-Nirenberg and the Caffarelli-Kohn-Nirenberg interpolation inequalities associated with Sobolev-Coulomb spaces for the (fractional) derivative $0 \le  s \le 1$.  These inequalities are relevant and useful in the study of  Thomas-FermiDirac-von Weizs\"acker models of density functional theory \cite{BBL81, LL05, Lieb81},   or  of Hartree-Fock theory \cite{CLL01, LS10, CDSS13},  or of Hardy-Lieb-Thirring type inequalities \cite{LS10}. 
 
The first part  of the paper is devoted to the following type inequality, for $g \in C^\infty_c(\mR^d)$,  
\begin{equation}\label{CSF1}
\nrm{g}_{L^\gamma(\ro^d)} \leq C \| g\|_{\dot W^{s, p}(\mR^d)}^{\beta_1 p} \bct{\int_{\ro^d}\int_{\ro^d} \frac{|g(x)|^q|g(y)|^q}{|x-y|^{d-\alpha}} \, dx \, dy}^{\beta_2}, 
\end{equation} 
for  $1< \gamma < + \infty$,  $1 \le p, \, q  < + \infty$, $0 \le s \le 1$, $0 < \alpha <  d$, $0 <   \beta_1, \beta_2 < + \infty $ under appropriate assumptions on these parameters, where $C$ is a positive constant independent of $g$. Here and in what follows, for any open set $\Omega \subset \mR^d$, we use the standard notation 
\begin{equation}
\nrm{g} _{\dot W^{s, p}(\Omega)} =  \left\{\begin{array}{cl} \dsp \bct{\int_{\Omega}\int_{\Omega} \frac{|g(x)-g(y)|^p}{|x-y|^{d+sp}}dxdy}^{\frac{1}{p}} & \mbox{ for } 0 < s < 1, \\[6pt]
\dsp \left( \int_{\Omega} |\nabla g (x)|^p \, dx \right)^{1/p} &  \mbox{ for } s = 1, \\[6pt]
\dsp \left( \int_{\Omega} |g (x)|^p \, dx \right)^{1/p} & \mbox{ for } s =0. 
\end{array}\right.
\end{equation}
One can easily check by scaling that \eqref{CSF1} holds only if the following two identities hold
\begin{align}\label{betaEquations}
\beta_1p+2\beta_2q =1 \quad \mbox{ and } \quad (d-sp)\beta_1+(d+\alpha)\beta_2 = d/ \gamma,
\end{align}

Mathematically, inequalities of type \eqref{CSF1} was first studied by Lions \cite{Lions81, pllCMP}. In his pioneering work,  he established 
\begin{align}\label{couSobLoc}
\nrm{g}_{L^3(\ro^3)} \leq C \nrm{\nabla g}^\frac{1}{3}_{L^2(\ro^3)} \bct{\int_{\ro^3} \int_{\ro^3} \frac{|g(x)|^2|g(y)|^2}{|x-y|}dx dy}^\frac{1}{6}  \quad  \forall g \in C_c^\infty(\ro^3).
\end{align}
in used it to study of Hartree-Fock equations. In the spirit of the famous Gagliardo-Nirenberg's inequalities \cite{Gagliardo59, Nirenberg59}, \eqref{couSobLoc} has been extended to a more general setting of the form \eqref{CSF1}.  Inequality \eqref{CSF1} is  so far only known for $p=2$ and $0 < s \le 1$. More precisely,   \eqref{CSF1} has been previously established  by Bellazzini, Frank, Visciglia \cite[Proposition 2.1]{BelFrVis} in the case $p = 2$, $q =2$, and $0 < s < 1$ (see also \cite[(21)]{LunNamPor} for the case $p=q=2$, $\alpha = d - 2 s$),  by  
Mercuri, Moroz, Van Schaftingen \cite{MerMorSch} in the case $p=2$ and $s=1$, and by Bellazzini, Ghimenti, Mercuri, Moroz, and Van Schaftingen \cite[Theorems 1.1 and 1.2]{MorSch'18} in the case $p=2$ and $0 < s < 1$.  The case $p=2$ is thus completely understood but only established very recently unless $s=0$. The case $p \neq 2$ is open {\it even} for $s=1$  and the case $s=0$ has not been considered previously.

\medskip 
The first main result of this paper is on the Gagliardo-Nirenberg interpolation type inequalities involving the Coulomb term. We have

\begin{thm}\label{thm1}
Let $d \ge 1$, $0 \le s \le 1$, $1 < \gamma < + \infty$,  $1 \leq  p, \, q < + \infty$, $0<\alpha < d$, and $0  <  \beta_1, \, \beta_2 < + \infty$. Assume \eqref{betaEquations} and the following fact
\begin{equation}\label{thm1-ass}
\beta_1 \gamma +\beta_2\gamma\geq 1. 
\end{equation}
Then \eqref{CSF1} holds for all $g \in L^1(\mR^d)$ with compact support \footnote{We use here the convention $+\infty \le + \infty$.}. 
\end{thm}

\begin{remark} \rm Condition~\eqref{thm1-ass} is optimal, see \Cref{rem-thm1-*}. 
\end{remark}

\begin{remark} \rm One can ask what happens in \Cref{thm1} if $\beta_1 = 0$ or $\beta_2 = 0$. In fact, \Cref{thm1} holds with $\beta_2 = 0$.  However, when $\beta_2 = 0$, \eqref{CSF1} is just the standard Sobolev inequality and there is nothing new. Concerning $\beta_1$, one can check that $\beta_1$ cannot be 0 since the assumption 
$$
(d-sp) \beta_1 + (d+ \alpha) \beta_2 = d/ \gamma 
$$
implies that if $\beta_1 = 0$ then
$$
\gamma \beta_2 = \frac{d}{(d+\alpha)} < 1: \mbox{ which contradicts to \eqref{thm1-ass}}. 
$$
\end{remark}

\begin{remark} \label{remark-thm1} \rm \Cref{thm1} also holds for $\gamma = 1$. In this case, since 
$$
\beta_1 \gamma + \beta_2 \gamma =1 \quad \mbox{ and } \quad 
(d-sp) \beta_1 + (d+ \alpha) \beta_2 = d, 
$$
it follows that $p = q = 1$, $s = 0$, and $\alpha = d$. Then \eqref{CSF1} is again just  the standard Gagliardo-Nirenberg inequality. 
\end{remark}

\begin{remark}  \rm \Cref{thm1} also holds for $\alpha = d$. Nevertheless, the conclusion in this case just follows from the the standard Gagliardo-Nirenberg inequalities. 
\end{remark}

\Cref{thm1} is known for the case $p=2$ with $s>0$ as mentioned above. Nevertheless, the known assumptions are stated in  a  more involved manner than condition \eqref{thm1-ass}. More precisely, when $p=2$, instead of \eqref{thm1-ass}, the conclusion of \Cref{thm1} was shown under the condition $\big($$(d+\alpha)- q(d- 2 s)\neq 0$ and \eqref{range1}$ \;  below \big)$ or $\big($$(d+\alpha)-q(d- 2s) = 0$ and \eqref{range2}$ \; below \big)$, where 
\begin{align}\label{range1}
\begin{cases}
\dsp \frac{2(\alpha+2qs)}{\alpha + 2 s} \leq \gamma \leq \frac{2d}{d-2s} &\text{ if } 2s<d \text{ and } (d+\alpha) - q(d- 2s)>0, \\[6pt]
\dsp \frac{2d}{d- 2s} \leq \gamma \leq \frac{2(\alpha+2qs)}{\alpha+2s} &\text{ if } 2s<d \text{ and } (d+\alpha) -q(d- 2s)<0, \\[6pt]
\dsp  \frac{2(\alpha+2qs)}{\alpha+2s}\leq \gamma <\infty  &\text{ if } 2s\geq d,
\end{cases}
\end{align}
and
\begin{equation}\label{range2}
\frac{\alpha(d-2 s)}{2d(\alpha+2s)} \leq \beta_1 < + \infty, \quad 0\leq \beta_2 \leq \frac{s(d-2s)}{d(\alpha+2s)}. 
\end{equation}
The proof is then given to case by case separately. One can check that  \eqref{thm1-ass} is equivalent to these conditions (see \Cref{sect-E}). In this paper, we present a new approach to obtain \Cref{thm1} in which condition \eqref{thm1-ass} appears very naturally.

\medskip  
We next extend \Cref{thm1} in the spirit of Caffarelli-Kohn-Nirenberg's inequalities. Caffarelli, Kohn, and Nirenberg  \cite{CKN} (see also \cite{CKN-original}) proved the following well-known inequality 
\begin{equation}\label{CaKN}
\|  |x|^{\tau'} g \|_{L^{\gamma'}(\R^d)} \le C \| |x|^\alpha \nabla g \|_{L^p(\R^d)}^{a} \| |x|^\beta g   \|_{L^q(\R^d)}^{(1-a)} \quad  \mbox{ for } g \in C^1_{c}(\R^d), 
\end{equation}
under appropriate assumptions of the parameters. This family of inequalities has been extended by Nguyen and Squassina \cite{HmnSqa} (see also \cite{Ng-Squ3}) for fractional Sobolev spaces where the quantity $\| |x|^\alpha \nabla u \|_{L^p(\R^d)}^p$ is replaced by 
$$
\int_{\mR^d} \int_{\mR^d} \frac{|g(x) - g(y)|^p |x|^{\alpha_1 p } |y|^{\alpha_2 p }}{|x - y|^{d + sp }} \, dx \, dy
$$
under appropriate assumptions of the parameters. Previous results can be found in \cite{AB17, FrSei, MS}.

\medskip 
In this direction, we establish our second main result of this paper on the Caffarelli-Kohn-Nirenberg interpolation inequalities associated with the Coulomb term.

\begin{thm}\label{thm2} Let $d\geq 1$, $0 \le    s  \le    1$,  $1\leq \gamma', \, p, \, q< + \infty$, $0< \alpha < d$, $0 <  \beta_1,  \, \beta_2 < + \infty$, $\tau', \alpha_{1, 1}, \alpha_{1, 2},  \, \alpha_{2, 1}, \alpha_{2, 2} \in \mR$. Set $\alpha_1 = \alpha_{1, 1} + \alpha_{1, 2}$ and $\alpha_{2} = \alpha_{2, 1} + \alpha_{2, 2}$, and   define $\sigma, \, \gamma$ by 
\begin{equation}\label{thm2-sigmagamma}
\sigma = \beta_1p \alpha_1  + \beta_2 q \alpha_2 \quad \mbox{ and } \quad (d-sp)\beta_1+(d+\alpha)\beta_2 =d/ \gamma. 
\end{equation}
Assume that 
\begin{align}\label{betaEquations-m}
\beta_1p+2\beta_2q =1, \quad \frac{1}{\gamma'}+\frac{\tau'}{d} = \frac{1}{\gamma} +\frac{\sigma}{d}, 
\end{align}
\begin{equation}\label{thm2-cd3}
\gamma \ge \gamma', \quad \gamma > 1, 
\end{equation}
and,  either 
\begin{equation}\label{thm2-cd4-1}
\beta_1 \gamma' +  \beta_2 \gamma' \ge 1, 
\end{equation}
or  
\begin{equation}\label{thm2-cd4-2}
\left( \frac{1}{p} (sp - d - \alpha_1 p)  + \frac{1}{2q}(\alpha + d + \alpha_2 q)  \neq 0 \quad \mbox{ and } \quad  \beta_1 \gamma +  \beta_2 \gamma >  1 \right). 
\end{equation}
Then, if either 
$$
 \frac{1}{\gamma'} +\frac{\tau'}{d}>0 \mbox{ and  $g\in L^1(\ro^d)$ with compact support,}
$$
or 
$$
\frac{1}{\gamma'} +\frac{\tau'}{d}<0  \mbox{ and $g\in L^1_{loc}(\mR^d)$ with $g = 0$ in a neighborhood of $0$}, 
$$
then it holds 
\begin{multline}\label{thm2-cl}
\bct{\int_{\ro^d} |g|^{\gamma'} |x|^{\tau'\gamma'}}^{\frac{1}{\gamma'}}
\leq  C \| g\|_{\dot W^{s, p, \alpha_{1, 1}, \alpha_{1, 2}}}^{p \beta_1}
\times 
\left(\int_{\mR^d}\int_{\mR^d} \frac{|g(x)|^q|g(y)|^q |x|^{\alpha_{2, 1}q} |y|^{\alpha_{2, 2}q}}{|x-y|^{d-\alpha}}  \, dx \, dy \right)^{\beta_2}, 
\end{multline}
where $C>0$ is a constant independent of $g$.
\end{thm}

Here and in what follows, we denote, for $t_1, t_2 \in \mR$, $1 \le p <  + \infty$, and $0 \le s \le 1$,  
\begin{equation}
\|g \|_{\dot W^{s, p, t_1, t_2}(\mR^d)}^p =
\left\{ \begin{array}{cl}
\dsp  \int_{\mR^d} \int_{\mR^d} \frac{|g(x) - g(y)|^p |x|^{ t_1 p } |y|^{ t_2  p } }{|x - y|^{d + sp }}  \, dx \, dy & \mbox{ for } 0 < s  < 1\\[6pt]
\dsp \int_{\mR^d} |g(x)|^p |x|^{(t_1 + t_2) p} \, dx  & \mbox{ for } s=0,\\[6pt]
\dsp  \int_{\mR^d} |\nabla g(x)|^p |x|^{ (t_1 + t_2) p} \, dx & \mbox{ for } s=1.  
 \end{array}\right. 
\end{equation}

\begin{remark}\label{CKN-ineq}
The case $\alpha=d$, corresponds to the usual Caffarelli-Kohn-Nirenberg inequalites, see \cite{HmnSqa} to compare the conditions satisfied by the parameters. The study of the best constant of such inequalities are in general highly non-trivial and of great interest, see e.g.,   \cite{DolEst'12, DolEstLos} and references therein.
\end{remark}

As an application of \Cref{thm2}, we derive a new family of one body interpolation inequalities and then use it to establish a new family of many body interpolation inequalities. These are our third main results of this paper and  presented later in \Cref{sect-P}.

\medskip 
We now review briefly several known approaches for \eqref{CSF1}.  One  is  based on the standard Gagliardo-Nirenberg inequality and the fractional chain rule, see e.g. \cite{Gatto}, and standard interpolation inequalities,  as in \cite{MorSch'18, BelFrVis}. Another approach is based on the  Hardy-Lieb-Thirring's  many body interpolation inequalities,  as in \cite{LunNamPor}.  These approaches use essentially the fact that $p=2$.

\medskip 
In this paper,  we propose a different approach to establish \Cref{thm1} and  \Cref{thm2}.  The proof of \Cref{thm1} follows closely the approach proposed by Nguyen  \cite{NgSob3} in his study Sobolev's inequalities associated with non-local, non-convex functionals. The idea is first to establish the corresponding Poincar\'e inequality. One then uses a covering argument to derive from it an estimate in $L^\gamma_w$ ($L^{\gamma}$-weak). The estimate in  $L^\gamma$ is then established via the estimate in $L^\gamma_w$
and involves the truncation technique due to Maz'ya and a result of the theory of maximal (sharp) functions due to Fefferman and Stein. The proof of \Cref{thm2} uses  similar arguments of Nguyen and Squassina \cite{HmnSqa} where the authors established the full range Cafffarelli-Kohn-Nirenberg's inequality for fractional Sobolev spaces. The idea is to first establish the conclusion of \Cref{thm2} under the assumption \eqref{thm2-cd4-1}. The starting point is the corresponding Gagliardo-Nirenberg type inequality established in \Cref{thm1} (see \Cref{domSob}). One then decomposes 
the space into annulus and  applies appropriately the corresponding Gagliardo-Nirenberg type inequality. This idea has its roots from harmonic analysis when the decomposition is given  for the  frequency variables. The proof in the case \eqref{thm2-cd4-2} is derived from \eqref{thm2-cd4-1} using a scaling argument. 

\medskip 

The organization of the paper  is as follows. \Cref{sect-GN} and  \Cref{sect-CKN} are devoted to the proof of \Cref{thm1} and  \Cref{thm2}, respectively. In \Cref{sect-GN} (\Cref{sect-E}), we also derive a connection between our conditions in \Cref{thm1} with the  known result forms.  In \Cref{sect-P}, we use \Cref{thm2} to establish a new family of  one body interpolation  inequalities. Another main ingredient of this proof is the sharp (fractional) Hardy inequalities with the remainder due to Frank and Seiringer \cite{FrSei}.  We then use this family of one body interpolation  inequalities to establish a new family of many body interpolation type inequalities  following the strategy of Lundholm, Nam,  and Portmann  \cite{LunNamPor}. These results are new even in the case $s=1$ and $p \neq 2$ and their proofs are new even in the case $p=2$.

\section{Gagliardo-Nirenberg interpolation  inequalities involving Coulomb terms - Proof of \Cref{thm1}} \label{sect-GN}

This section is devoted to the proof of \Cref{thm1}. The proof is in the spirit of the one given in \cite{NgSob3} where the author established new Poincar\'e and Sobolev inequalities related to new characterizations of Sobolev spaces via non-local, non-convex terms in \cite{NgSob1, BourgNg, NgSob2} (see also \cite{BrezisNg, BreHmn, MallickNg} for related topics).  The ideas  of the proof are as follows. We first derive a Poincar\'e inequality involving Coulomb terms, which is almost free in this context. The integrability desired is then established via Vitali's covering lemma and the truncation method, which has its root in the work of Mazda \cite{Ma}.  This part also uses an interesting result of the theory of sharp functions due to Fefferman and Stein \cite{FS}.  This section containing three subsections is organized as follows. Several useful lemmas are presented in the first subsection. The proof of \Cref{thm1} is then  given in the second one. In the last subsection, we discuss other forms of the assumptions of \Cref{thm1}.

\subsection{Preliminaries}

For $D$ a measurable set of $\mR^d$ and $g \in L^1(D)$, denote $|D|$ its Lebesgue mesure and 
$$
(g)_D = \fint_D g(y) \,  dy: = \frac{1}{|D|} \int_{D} g(y) \, dy. 
$$

We begin with  a simple but useful version of Poincar\'e  inequality involving a Coulombian term. 

 \begin{lem}\label{lem-Poincare}
 Let $d \ge 1$, $0 \le s \le 1$, $1 < \gamma < + \infty$, $1 \leq p, \, q < + \infty$, $0<\alpha <  d$, and $0 < \beta_1, \, \beta_2 < + \infty$, and assume \eqref{betaEquations}, and let $B$ be an open ball or an open cube in $\ro^d$.   We have,  for $u\in L^1(B)$, 
 \begin{align*}
 \fint_{B}\md{u-(u)_B} \, dx  \leq \frac{C}{\md{B}^{\frac{1}{\gamma}}} \nrm{u}_{\wsp(B)}^{\beta_1p} \bct{\int_B\int_B \frac{|u(x)|^q|u(y)|^q}{|x-y|^{d-\alpha}}dxdy}^{\beta_2},
 \end{align*}
 where $C>0$ is a constant independent of $u$ and $B$. 
 \end{lem}

\begin{proof} 
It is easy to check that  
\begin{align}\label{p'care1}
\fint_B\md{u-(u)_B} \, dx  \le \fint_B \fint_{B} |u(x) - u(y)| \, dx \, dy  \leq \frac{C}{\md{B}^\frac{d-sp}{pd}} \nrm{u}_{\wsp(B)}
\end{align}
and
\begin{multline}\label{p'care2}
\fint_B\md{u-(u)_B} \, dx \leq 2 \fint_B |u| \, dx   \leq \frac{2}{\md{B}^\frac{1}{q}} \bct{\int_B |u(x)|^q dx}^\frac{1}{q} \\[6pt] \leq \frac{C}{\md{B}^\frac{d+\alpha}{2qd}} \bct{\int_B\int_B \frac{|u(x)|^q|u(y)|^q}{\md{x-y}^{d-\alpha}}dxdy}^\frac{1}{2q}.
\end{multline}
Here and in what follows  in this proof, $C$ denotes a positive constant independent of $u$ and $B$. 

Since $\beta_1, \beta_2 >  0$ and $p\beta_1 +2q\beta_2=1$ by  \eqref{betaEquations}, we have 
\begin{multline}\label{p'care-p1}
\fint_B|u-(u)_B|  =  \bct{\fint_B \md{u-(u)_B}}^{p\beta_1} \bct{\fint_B \md{u-(u)_B}}^{2q\beta_2} \\[6pt]
\mathop{\leq}^{\eqref{p'care1}, \eqref{p'care2}}  \frac{C}{\md{B}^{\frac{(d-sp)\beta_1}{d}+ \frac{(d+\alpha) \beta_2}{d}}} \nrm{u}_{\wsp(B)}^{\beta_1p} \bct{\int_B\int_B \frac{|u(x)|^q|u(y)|^q}{|x-y|^{d-\alpha}}dxdy}^{\beta_2}.
\end{multline}
Since 
$$
\frac{1}{\gamma} \mathop{=}^{\eqref{betaEquations}} \frac{(d-sp)\beta_1}{d}+ \frac{(d+\alpha) \beta_2}{d}, 
$$
the conclusion follows from \eqref{p'care-p1}. 
\end{proof}

We next present a variant of \cite[Lemma 7]{NgSob3},  whose proof is based on Vitali's covering lemma.

\begin{lem}\label{coveringLem}
Let $d \ge 1$, $1 <   \gamma < + \infty$, and $0 <  \beta_1, \, \beta_2 < + \infty$ be such that $\beta_1 \gamma + \beta_2 \gamma \ge 1 $. Let $h_1, h_2 \in L^1(\mR^d)$, and let $g$ be a measurable function defined in $\mR^d$. Assume that 
\begin{equation}\label{coveringLem1}
g(x) \leq \sup_B \frac{1}{\md{B}^\frac{1}{\gamma}} \bct{\int_B |h_1|}^{\beta_1} \bct{\int_B |h_2|}^{\beta_2} \mbox{ for a.e. } x\in \ro^d, 
\end{equation}
where the supremum is taken over all open balls (or open cubes) containing  $x$. Then
\begin{equation}\label{coveringLem-st1}
t^\gamma \big| \big\{ g>t \big\} \big| \leq C \nrm{h_1}^{\beta_1\gamma}_{L^1(\ro^d)} \nrm{h_2}^{\beta_2\gamma}_{L^1(\ro^d)}  \ \forall t>0,
\end{equation}
for some positive constant $C$ depending only on $\gamma, \beta_1, \beta_2$,  and $d$.
\end{lem}

\begin{proof} We only consider where the case the supremum is taken over all open balls containing  $x$. The case where the supremum is taken over all open cubes containing  $x$ follows in the same lines and is omitted. 

Fix $t > 0$. From \eqref{coveringLem1}, for a.e. $y\in \Sb{g>t}$,  there exists an open ball $B_y$ containing $y$ such that
\begin{equation*}
t \leq \frac{2}{\md{B_y}^\frac{1}{\gamma}} \bct{\int_{B_y}|h_1| \, dx}^{\beta_1} \bct{\int_{B_y}|h_2| \, dx}^{\beta_2}, 
\end{equation*}
which yields
\begin{equation*}
 \md{B_y} \leq \frac{2}{t^\gamma} \bct{\int_{B_y}|h_1| \, dx }^{\beta_1\gamma} \bct{\int_{B_y}|h_2| \, dx}^{\beta_2\gamma}. 
\end{equation*}
By applying Vittali's covering lemma to the set $\big\{g > t \big\}$ and to the family of open ball $B_y$,  there exists a countable  collection of mutually disjoint  open balls $(B_i)$ such that outside a set of zero measure, \footnote{Here $5 B_i$ denotes the open ball with the same center as $B_i$ but 5 times radius.}
\begin{equation}\label{coveringLem-p1}
\Sb{g>t} \subset \cup_i 5B_i,
\end{equation}
and 
\begin{equation}\label{coveringLem2}
\md{B_i} \leq \frac{C}{t^\gamma} \bct{\int_{B_i}|h_1|}^{\beta_1\gamma} \bct{\int_{B_i}|h_2|}^{\beta_2\gamma}.
\end{equation}
Applying \Cref{lem-Ineq} below  with $\tau= \beta_1\gamma$ and $\eta= \beta_2\gamma$,  after noting that $\beta_1 \gamma + \beta_2 \gamma \ge 1$,  we have 
  \begin{multline*}
  \md{\Sb{g>t}}  \mathop{\le}^{\eqref{coveringLem-p1}} C \sum_{i=1}^\infty |B_i|
 \mathop{\le}^{\eqref{coveringLem2}}  \frac{C}{t^\gamma} \sum_i \bct{\int_{B_i}|h_1|}^{\beta_1\gamma} \bct{\int_{B_i}|h_2|}^{\beta_2\gamma}  \\[6pt]
 \mathop{\leq}^{\Cref{lem-Ineq}} \frac{C}{t^\gamma} \bct{\int_{\ro^d}|h_1|}^{\beta_1\gamma} \bct{\int_{\ro^d}|h_2|}^{\beta_2\gamma},  
  \end{multline*}
which is the conclusion. 
\end{proof}

\begin{remark} \rm The proof of \Cref{lem-Poincare} and  \Cref{coveringLem} also work for $\gamma = 1$. Nevertheless, these results are later only applied to the case $\gamma > 1$. 
\end{remark}

The following simple result, whose proof is omitted, is used in the proof of  \Cref{coveringLem}. 

\begin{lem}\label{lem-Ineq} For $\tau, \eta > 0$ with $\tau + \eta \ge 1$, we have 
\begin{equation}
\sum_{i=1}^\infty   |a_{i}|^{\tau} |b_{i}|^{\eta}    \leq  \bct{\sum_{i=1}^\infty |a_{i}|}^{\tau} \bct{\sum_{i=1}^\infty |a_{i}|}^{\eta}  \mbox{ for } a_{i} \, b_i   \in \mR.
\end{equation}
\end{lem}

 As a consequence of \eqref{coveringLem-st1}, one derives that $g \in L^{\gamma}_{w}(\mR^d)$ ($L^{\gamma}$-weak) if $g$ is non-negative and satisfies \eqref{coveringLem1}. We next present a variant of \Cref{coveringLem} which deals with the $L^\gamma$-integrability of $g$ instead of $L^\gamma$-weak integrability.  This variant inspired by \cite[Lemma 8]{NgSob3} is the main ingredient of the proof of \Cref{thm1}.   To this end, we first recall the definition of the dyadic maximal functions and the dyadic sharp maximal functions, see, e.g.,  \cite{StHar}.

\begin{defi}\label{srpFuncs}
Let $g\in L^1_{loc}(\ro^d)$. The dyadic maximal function $M^\Delta g$ and the dyadic sharp maximal function $g^{\sharp, \Delta}$ are defined as follows
$$\bct{M^\Delta g}(x): = \sup_{Q} \fint_Q |g|dy,$$ and $$g^{\sharp,\Delta}(x) := \sup_{Q} \fint_Q|g-(g)_Q| dy,$$ where the supremum is taken over all dyadic cubes $Q$ containing $x$.
\end{defi}

The following definition is also useful. 

\begin{defi}\label{def-Truncation}
 For each $k\in \mathbb{Z}$ and a {\it non-negative} function $g$ defined in $\mR^d$, define the truncation operator 
 \begin{align*}
 \cT_k(g)(x)= \left\{\begin{array}{cl}
	10^{k+1}-10^k &\text{ if } x\in \Sb{g>10^{k+1}}, \\[6pt] 
	g -10^k &\text{ if } x\in \Sb{10^k<g\leq 10^{k+1}}, \\[6pt]
	 0 &\text{ if } x\in \Sb{g\leq10^k}. 	 
 	\end{array} \right.
 \end{align*}
\end{defi}

We are ready to present a variant of \Cref{coveringLem} concerning $L^\gamma$-integrability. 

\begin{lem}\label{lem-key-thm1}
Let $d \ge 1$, $1 < \gamma  < + \infty$,  and $0  <  \beta_1, \, \beta_2 < + \infty$ be such that 
\begin{equation}\label{lem-key-thm1-cd0}
\beta_1 \gamma +\beta_2\gamma\geq 1. 
\end{equation}
Let  $g \in L^1(\ro^d)$ with $|\{|g| > 0 \}| <  + \infty$, and set 
\begin{equation}\label{lem-key-thm1-def-gk}
g_k = \cT_k (g) \mbox{ for } k \in \mZ. 
\end{equation}
Assume that there exist two sequences $(h_{1, k} ), (h_{2, k}) \subset L^1(\ro^d)$,  and two non-negative functions $h_1, \, h_2 \in L^1(\ro^d)$ such that, for $t > 0$ and $k \in \mZ$,  
\begin{equation}\label{lem-key-thm1-cd1}
\md{\Sb{g_k^{\sharp,\Delta}>t}} \leq \frac{1}{t^\gamma}  \nrm{h_{1, k}}^{\beta_1\gamma}_{L^1(\ro^d)}  \nrm{h_{2, k}}^{\beta_2\gamma}_{L^1(\ro^d)} , 
\end{equation}
and, for $j=1, 2$,  
\begin{equation}\label{lem-key-thm1-st2}
 \sum _{k=1}^\infty \md{h_{j, k}} \leq h_j \mbox{ in } \mR^d.
\end{equation}
Then $g \in L^\gamma(\mR^d)$ and
\begin{align*}
\|g \|_{L^\gamma(\mR^d)} \leq C \nrm{h_1}^{\beta_1}_{L^1(\ro^d)}  \nrm{h_2}^{\beta_2}_{L^1(\ro^d)}
\end{align*}
for some positive constant $C$ independent of $g, \, h_1$,  and $h_2$.
\end{lem}

\begin{proof}
Let $0<b<1$, $c>0$, and $f\in L^1_{loc}(\ro^d)$. We recall that, see, e.g., \cite[Estimate (22), p.153]{StHar}, 
\begin{align}\label{weak-strngEst1}
\md{\Sb{M^\Delta f>\alpha, f^{\sharp,\Delta}\leq c\alpha}} \leq \frac{2^dc}{1-b} \md{\Sb{M^\Delta f>b\alpha}}  \ \forall \alpha>0.
\end{align}

Applying \eqref{weak-strngEst1} with $f=g_k$, $b=\frac{1}{10}$, $\alpha=10^k$,  and $0<c<\frac{1}{2}$  (to be chosen later), we have 
\begin{align*}
\md{\Sb{M^\Delta g_k>10^k}} \leq c2^{d+1} \md{\Sb{M^\Delta g_k>10^{k-1}}} + \md{\Sb{ g^{\sharp,\Delta}_k>c10^k}}.
\end{align*}
This yields,  for any $m,n\in \mathbb{Z}$ with $n \ge m +1$, 
\begin{multline}\label{weak-strngEst2}
\sum_{k =m} ^n 10^{k\gamma}\md{\Sb{M^\Delta g_k>10^k}} \\[6pt] 
\leq c2^{d+1} \sum_{k = m}^n 10^{k\gamma}\md{\Sb{M^\Delta g_k>10^{k-1}}} + \sum_{k = m}^n 10^{k\gamma}\md{\Sb{ g^{\sharp,\Delta}_k>c10^k}}.
\end{multline}

We first derive an lower bound for the LHS of \eqref{weak-strngEst2}. From the definition of $g_k$ and the fact $\Sb{M^\Delta g_k>10^k}\supset \Sb{g_k>10^k}$,  we have, for $n \ge m+1$, 
\begin{align}\label{weak-strngEst3}
\sum_{k=m}^n 10^{k\gamma}\md{\Sb{M^\Delta g_k>10^k}} \geq C\int_{10^{m+1}}^{10^{n+2}} t^{\gamma-1} \md{\Sb{|g|>t}} dt. 
\end{align}
Here and in what follows in this proof, $C>0$ denotes a constant independent of $g$, $h_1$, $h_2$,  and $k, \, m, \, n$. 

We next derive an upper bound of the RHS of \eqref{weak-strngEst2}. By the theory of maximal functions, we deduce from  the definition of $g_k$ in \eqref{lem-key-thm1-def-gk} that, for $n \ge m+1$,  
\begin{equation}\label{weak-strngEst4}
\sum_{k = m}^n 10^{k\gamma}\md{\Sb{M^\Delta g_k>10^{k-1}}} \leq C \sum_{k=m}^n\int_{\ro^d} |g_k|^\gamma dx \leq C \int_{10^m}^{10^{n+2}} t^{\gamma-1}\md{\Sb{|g|>t}}dt.
\end{equation}
We also have
 \begin{equation}\label{weak-strngEst5}
 \sum_m^n 10^{k\gamma}\md{\Sb{ g^{\sharp,\Delta}_k>c10^k}} \mathop{\leq}^{\eqref{lem-key-thm1-cd1}} \frac{1}{c^\gamma} \sum_m^n \nrm{h_k^1}^{\beta_1\gamma}_{L^1(\ro^d)}
\nrm{h_k^2}^{\beta_2\gamma}_{L^1(\ro^d)} \mathop{\leq}^{\eqref{lem-key-thm1-cd0}} \frac{1}{c^\gamma} \nrm{h_1}^{\beta_1\gamma}_{L^1(\ro^d)}\nrm{h_2}^{\beta_2\gamma}_{L^1(\ro^d)}.
\end{equation}
Plugging the estimates \eqref{weak-strngEst3}, \eqref{weak-strngEst4} and \eqref{weak-strngEst5} into \eqref{weak-strngEst2} we obtain 
\begin{equation*}
\int_{10^{m+1}}^{10^{n+2}} t^{\gamma-1}\md{\Sb{|g|>t}}dt \leq C\bct{c2^{d+1}\int_{10^m}^{10^{n+2}} t^{\gamma-1}\md{\Sb{|g|>t}}dt + \frac{1}{c^\gamma} \nrm{h_1}^{\beta_1\gamma}_{L^1(\ro^d)}\nrm{h_2}^{\beta_2\gamma}_{L^1(\ro^d)} }. 
\end{equation*}
Choosing $c$ small enough so that $Cc2^{d+1}=\frac{1}{2}$, we have
\begin{equation*}
\int_{10^{m+1}}^{10^{n+2}} t^{\gamma-1}\md{\Sb{|g|>t}}dt \leq \int_{10^m}^{10^{m+1}} t^{\gamma-1}\md{\Sb{|g|>t}}dt + C\nrm{h_1}^{\beta_1\gamma}_{L^1(\ro^d)}\nrm{h_2}^{\beta_2\gamma}_{L^1(\ro^d)} .
\end{equation*}
Letting first $n \to + \infty$, then $m \to - \infty$, and noting that  $\md{\Sb{|g|>0}}< + \infty$, we obtain 
\begin{equation*}
\int_{\ro^d} |g(x)|^\gamma \, dx \leq C \nrm{h_1}^{\beta_1\gamma}_{L^1(\ro^d)}\nrm{h_2}^{\beta_2\gamma}_{L^1(\ro^d)},
\end{equation*}
which implies the conclusion. The proof is complete
\end{proof}

\subsection{Proof of Theorem \ref{thm1}.}  The proof follows from \Cref{lem-Poincare}, \Cref{coveringLem}, and \Cref{lem-key-thm1} as in \cite{NgSob3}. 
Set 
\begin{equation}\label{lem-key-thm1-def-gk}
g_k = \cT_k (|g|) \mbox{ for } k \in \mZ, 
\end{equation}
where $\cT_k$ is given by \Cref{def-Truncation}.  Define, for $x\in \mR^d$, 
\begin{equation*}
h_1(x) = 
\left\{\begin{array}{cl}
|g(x)|^p  &  \mbox { for } s=0,\\[6pt]
\dsp \int_{\ro^d} \frac{|g(x)-g(y)|^p}{|x-y|^{d+sp}} \, dy &  \mbox{ for } s \in (0, 1), \\[12pt] 
\md{\nabla g(x)}^p &  \mbox{ for } s=1, 
\end{array}\right.
\end{equation*}
\begin{equation*}
h_{1, k}(x) = 
\left\{\begin{array}{cl}
|g_k(x)|^p & \mbox { for } s=0,\\[6pt]
\dsp \int_{\ro^d} \frac{|g_k(x)-g_k(y)|^p}{|x-y|^{d+sp}} \, dy & \mbox{ for } s \in (0, 1), \\[12pt] 
\md{\nabla g_k(x)}^p & \mbox{ for } s=1,    
\end{array}\right.
\end{equation*}
\begin{equation*}
 h_2(x) := |g(x)|^q\int_{\ro^d} \frac{|g(y)|^q}{|x-y|^{d-\alpha}} \, dy,  \quad \mbox{ and }  \quad h_{2, k} (x):= |g(x)|^q\int_{\ro^d} \frac{|g_k(y)|^q}{|x-y|^{d-\alpha}} \, dy.
\end{equation*}

We claim that, for $j =1, \, 2$,   
\begin{equation}\label{thm1-claim1}
\sum _{k \in \mathbb{Z}} \md{h_{j, k}} \leq h_j \quad \mbox{ and } \quad   h_j \in L^1(\ro^d). 
\end{equation}

We admit Claim \eqref{thm1-claim1} and continue the proof. 

By \Cref{lem-Poincare}, we have 
\begin{align*}
g_k^{\sharp,\Delta}(x) \leq C \sup_{Q} \frac{1}{\md{Q}^\frac{1}{\gamma}} \bct{\int_Q |h_1^k|}^{\beta_1} \bct{\int_Q |h_2^k|}^{\beta_2},
\end{align*}
where the supremum is taken over all cubes $Q$ containing $x\in \ro^d$.
Applying \Cref{coveringLem}, we obtain, for $k \in \mZ$,  
\begin{align*}
\md{\Sb{g_k^{\sharp,\Delta}>t}} & \leq \frac{C}{t^\gamma}  \nrm{h_k^1}^{\beta_1\gamma}_{L^1(\ro^d)}  \nrm{h_k^2}^{\beta_2\gamma}_{L^1(\ro^d)}  \mbox{ for } t>0. 
\end{align*}
The conclusion now follows from \Cref{lem-key-thm1}.

\medskip 

It remains to prove Claim \eqref{thm1-claim1}.  We first establish Claim \eqref{thm1-claim1} with $j=1$. Claim \eqref{thm1-claim1} with $j=1$ is clear for $s=0$ and $s=1$. Claim \eqref{thm1-claim1} with $j=1$ in the case $s \in (0, 1)$ follows from the fact $p \ge 1$ and 
\begin{equation}\label{thm1-claim2}
\sum_{k\in\mathbb{Z}} |g_k(x)-g_k(y)| \leq |g(x)-g(y)|, 
\end{equation}
which can be proved as follows. 

We first deal with the case $|g(x)| \neq 0$ and $|g(y)| \neq 0$. 
Let $m, \, n \in \mZ$ be such that  
$$ 
10^m < |g(x)| \leq 10^{m+1} \quad \mbox{ and } \quad  10^n <|g(y)| \leq 10^{n+1}. 
$$
Without loss of generality, one might assume that $|g(y)| \ge |g(x)|$ and this in turn implies $n \ge m$.  

We have, for $k \in \mZ$,  
\begin{equation*}
g_k(x) = \left\{\begin{array}{cl}
0 & \mbox{ for }   k \geq  m+1, \\[6pt]
|g(x)| - 10^m  & \mbox{ for } k=m, \\[6pt]
10^{k+1}-10^k &  \mbox{ for }  k<m, 
\end{array}\right.  
\quad  \text{and} \quad
g_k(y) = \left\{\begin{array}{cl}
0 & \mbox{ for }   k \geq  n +1, \\[6pt]
|g(y)| - 10^n & \mbox{ for } k = n, \\[6pt]
10^{k+1}-10^k &  \mbox{ for }  k < n.
\end{array}\right.  
\end{equation*}
This yields, if $n \ge m + 1$,  
\begin{multline*}
\sum_{k\in \mathbb{Z}} |g_k(x)-g_k(y)| = \sum_{k=-\infty}^{m} |g_{k}(x)-g_k(y)| + \sum_{k=m+1}^l |g_k(y)| \\[6pt] 
 = \big| |g(x)| - 10^{m+1} \big|+ \sum_{k=m+1} ^{l -1}(10^{k+1}-10^k) + |g(y)|-10^n = |g(y)| - |g(x)|     \le |g(x)-g(y)|, 
\end{multline*}
and, if $n = m $, 
\begin{equation*}
\sum_{k\in \mathbb{Z}} |g_k(x)-g_k(y)| =  |g_m(x) - g_m(y)| =  |g(y)| - |g(x)|    \le |g(x)-g(y)|. 
\end{equation*}
Hence Assertion \eqref{thm1-claim2}  is proved in this case. We next deal with 
Assertion \eqref{thm1-claim2} in the case $|g(x)| = 0$ or $|g(y)| = 0$. This follows from the fact, for $z \in \mR^d$,  
\begin{equation}\label{thm1-claim3}
\sum_{k\in\mathbb{Z}} |g_k(z)| = |g(z)|. 
\end{equation}

Claim \eqref{thm1-claim1} with $j=2$ is just a consequence of \eqref{thm1-claim3} and the fact $q \ge 1$. 

\medskip 
The proof is complete. \qed

\begin{remark}\label{rem-thm1-*}
It is worth noting that \eqref{CSF1} is false if  \eqref{thm1-ass} does not hold, which shows the optimality of \eqref{thm1-ass}. Indeed,  let $a\in \mR^d\setminus\{0\}$ and  $\eta \in C^\infty(\ro^d)\setminus\{0\} $ be such that $0 \le \eta \le 1 $, $\eta = 1$ in $B_{1/8}$, and  $\mbox{supp } \eta \subset B_{1/4}$. For $m\in \mathbb{N}$, define 
$$
v_{m,a}(x)= \sum_{k=1}^m \eta(x+ka)  \mbox{ for }  \ x\in \ro^d.
$$ 
Then $v_{m,a}\in C_c^\infty(\ro^d)$, and for $|a|\to \infty$ we have 
$$
\| v_{m,a}\|_{L^\gamma} \ge C m^{1/ \gamma}, 
$$
$$
\| v_{m,a}\|_{\dot W^{s, p}(\mR^d)}^{\beta_1 p} \le  C m^{\beta_1}, \quad \mbox{ and } \quad  \bct{\int_{\ro^d}\int_{\ro^d} \frac{|v_{m,a}(x)|^q|v_{m,a}(y)|^q}{|x-y|^{d-\alpha}} \, dx \, dy}^{\beta_2} \leq C m^{\beta_2}.  
$$
Here $C$ denotes a positive constant independent of $m$. 
Thus if  \eqref{CSF1} holds then $m  \le C m^{\beta_1\gamma+\beta_2\gamma}.$  This proves the optimality of \eqref{thm1-ass}.
\end{remark}

\subsection{Other forms of the assumptions of the Gagliardo-Nirenberg interpolation inequalities} \label{sect-E}

In this section, we give another form of condition \eqref{thm1-ass}, which is often found in the literature. We begin with (compare with \eqref{range1} for $p=2$)
 
 \begin{lem}\label{pro1-thm} Let $d \ge 1$, $0 \le s \le 1$, $1 \leq \gamma, \, p, \, q <\infty$, $0<\alpha<d$, and $0 \le \beta_1, \, \beta_2 < + \infty$, and assume \eqref{betaEquations}, and $(d+\alpha)p-2q(d-sp)\neq 0$.  Then \eqref{thm1-ass} is equivalent to the fact 
\begin{align}\label{gammaRange-proof}
\begin{cases}
\dsp \frac{p(\alpha+2qs)}{\alpha+sp} \leq \gamma \leq \frac{pd}{d-sp} &\text{ if } sp<d \text{ and } (d+\alpha)p-2q(d-sp)>0, \\[6pt]
\dsp \frac{pd}{d-sp} \leq \gamma \leq  \frac{p(\alpha+2qs)}{\alpha+sp} &\text{ if } sp<d \text{ and } (d+\alpha)p-2q(d-sp)<0, \\[6pt]
\dsp  \frac{p(\alpha+2qs)}{\alpha+sp}\leq \gamma <\infty  &\text{ if } sp\geq d.
\end{cases}
\end{align}
\end{lem}

\begin{proof}
Since $(d+\alpha)p-2q(d-sp)\neq 0$,  it follows from \eqref{betaEquations} that 
 \begin{equation}\label{betaValue-proof}
 \beta_1 = \frac{\gamma(d+\alpha)-2qd}{\gamma\bct{p(d+\alpha)-2q(d-sp)}} \quad \mbox{ and } \quad \beta_2 = \frac{pd-\gamma(d-sp)}{\gamma\bct{p(d+\alpha)-2q(d-sp)}}.
 \end{equation}
We then have
 \begin{align*}
   \beta_1 \gamma +\beta_2 \gamma -1&=  \frac{1}{p(d+\alpha)- 2q(d-sp)} \left[\gamma(d+\alpha)- 2qd + pd-\gamma(d-sp) -p(d+\alpha) +2q(d-sp)\right]  \\[6pt] &=  \frac{\alpha+sp}{p(d+\alpha)- 2q(d-sp)}\left[\gamma-\frac{p(\alpha+2qs)}{\alpha+sp} \right].
 \end{align*}
Since $\beta_1$, $\beta_2 >  0$,  we derive the equivalence of \eqref{thm1-ass} and \eqref{gammaRange-proof}.
\end{proof}

We next establish (compare with \eqref{range2} for $p=2$)

\begin{lem}\label{pro2-thm1} Let $d \ge 1$, $0 \le s \le 1$, $1 \leq \gamma, \, p, \, q <\infty$, $0<\alpha<d$, and $0 \le \beta_1, \, \beta_2 < + \infty$, and assume \eqref{betaEquations}, and $(d+\alpha)p-2q(d-sp) =  0$.  Then \eqref{thm1-ass} holds iff the following two conditions hold 
\begin{equation}\label{pro2-thm1-proof-eq1}
\frac{\alpha(d-sp)}{pd(\alpha+sp)} \leq \beta_1 < + \infty, \quad 0\leq \beta_2 \leq \frac{s(d-sp)}{d(\alpha+sp)}. 
\end{equation}
\end{lem}

\begin{proof}
Since $(d+\alpha)p-2q(d-sp) =  0$, it follows that  $sp<d$ and $q= \frac{p(d+\alpha)}{2(d-sp)}$. By \eqref{betaEquations}, we have 
$$(d-sp)\beta_1 +(d+\alpha) \beta_2  =\frac{d}{\gamma} \quad \mbox{ and } \quad \beta_1 +\frac{d+\alpha}{d-sp}\beta_2 = \frac{1}{p}.$$ 
This yields 
$$
\gamma= \frac{pd}{d-sp}. 
$$
We then have 
\begin{align*}
\beta_1\gamma+ \beta_2\gamma -1 &=  \frac{\beta_1 pd}{d-sp}-\frac{\beta_1pd}{d+\alpha} +\frac{d}{d+\alpha}-1 = \frac{\beta_1 (\alpha+sp)pd}{(d-sp)(d+\alpha)} -\frac{\alpha}{d+\alpha}.
\end{align*}
Therefore $\beta_1\gamma+\beta_2\gamma \geq 1$ if and only if $\beta_1\geq \frac{\alpha(d-sp)}{pd(\alpha+sp)}$. Since $\beta_1 +\frac{d+\alpha}{d-sp}\beta_2 = \frac{1}{p}$,   we derive that $\beta_1\geq \frac{\alpha(d-sp)}{pd(\alpha+sp)}$  if and only if  \eqref{pro2-thm1-proof-eq1} holds. 
\end{proof}

\section{Caffarelli-Kohn-Nirenberg interpolation type inequalities involving Coulomb terms - Proof of \Cref{thm2}}\label{sect-CKN}

The main goal of this section is to prove Theorem \ref{thm2}. We closely follow the techniques introduced by Nguyen and Squassina in \cite{HmnSqa} (see also \cite{Ng-Squ3}) to derive our results. We begin with a consequence of \Cref{thm1}. 

\begin{lem}\label{domSob}
Let $d\geq 1$, $0 \le  s \le  1$,  $1 < \gamma < + \infty$, $1\leq \gamma', \, p, \, q < + \infty$, $0<\alpha<d$, and $0 <  \beta_1,  \, \beta_2 < + \infty$.  Assume that  \eqref{betaEquations} and   \eqref{thm1-ass} hold, and 
$$
\gamma \ge \gamma'. 
$$ 
Let $\lambda>0$ and $0<r<R$,  and set 
$$
D_\lambda := \Sb{x\in \ro^d: \lambda r <|x|<\lambda R}.
$$
We have, for $g \in L^1(D_{\lambda})$, 
\begin{align}
\bct{\fint_{D_\lambda} |g-(g)_{D_{\lambda}}|^{\gamma'}}^{\frac{1}{\gamma'}}\notag \leq  \frac{C}{\lambda^{\frac{d}{\gamma}}} \nrm{g}_{\wsp(D_\lambda)}^{\beta_1p} \bct{\int_{D_\lambda}\int_{D_\lambda} \frac{|g(x)|^q|g(y)|^q}{|x-y|^{d-\alpha}} \, dx \, dy}^{\beta_2}, 
\end{align}
for some positive constant $C$ independent of $\lambda$ and $g$. 
\end{lem}

\begin{proof} 
Using \eqref{betaEquations}, by scaling we can assume that $\lambda=1$. \Cref{domSob} is now a consequence of \Cref{thm1}. 
\end{proof}

In the next two subsections, we present the proof $(i)$ and $(ii)$ of  Theorem \ref{thm2}, respectively. 

\subsection{Proof of $(i)$ of  Theorem \ref{thm2}} We are ready to give the proof of \Cref{thm2}. We closely follow  the strategy in \cite{HmnSqa}. We only consider the case $0 < s < 1$, the proof in general case follows similarly and is omitted. 

\medskip 

 The proof is divided into two steps. 

\begin{itemize}
\item Step 1: We establish $(i)$ of  Theorem \ref{thm2} assuming \eqref{thm2-cd4-1}. 

\item Step 2: We establish $(i)$ of  Theorem \ref{thm2} assuming  \eqref{thm2-cd4-2} and $\gamma' < \gamma$.

\end{itemize}

It is clear that  Assertion $(i)$ then follows from Steps 1 and  2. 

\medskip 
We now proceed Steps 1 and 2.

\medskip 
\noindent \textit{Step 1:}  We establish $(i)$ of  Theorem \ref{thm2} assuming \eqref{thm2-cd4-1}. 

Set 
$$
\mcal{A}_k : = \Sb{x\in \ro^d: 2^k\leq |x| <2^ {k+1}}.
$$ 
By \Cref{domSob}, we derive from \eqref{thm2-cd3}  that 
\begin{multline}\label{thm2-p1}
\bct{\fint_{\mcal{A}_k} \left|g-(g)_{\mcal{A}_k}\right|^{\gamma'}}^{\frac{1}{\gamma'}}  \leq  \frac{C}{2^{\frac{dk}{\gamma} }} \bct{\int_{\mcal{A}_k}\int_{\mcal{A}_k} \frac{|g(x)-g(y)|^p}{|x-y|^{d+sp}}dxdy}^{\beta_1} \\[6pt] 
\times \bct{\int_{\mcal{A}_k}\int_{\mcal{A}_k} \frac{|g(x)|^q|g(y)|^q}{|x-y|^{d-\alpha}}dxdy}^{\beta_2}. 
\end{multline}
Here and in what follows in the proof of \Cref{thm2}, $C$ denotes a positive constant independent of $g$ and $k$ (and also independent of $m$, and $n$, which appear later).  Since 
$$
2^{ \tau'\gamma' k} \int_{\mcal{A}_k} |g|^{\gamma'}  \le C  2^{(\tau' \gamma'+d)k} \fint_{\mcal{A}_k} \left|g-(g)_{\mcal{A}_k}\right|^{\gamma'} + C 2^{(\tau' \gamma'+d)k}\md{\fint_{\mcal{A}_k}g}^{\gamma'}, 
$$
using condition \eqref{betaEquations-m} and the definition of $\sigma$, we derive that 
\begin{multline}\label{sameAnnEst}
 \int_{\mcal{A}_k} |g|^{\gamma'} |x|^{\tau' \gamma'} \, dx   \leq  C 2^{(\gamma'\tau'+d)k}\md{\fint_{\mcal{A}_k}g}^{\gamma'}  \\[6pt]
 + C \bct{\int_{\mcal{A}_k}\int_{\mcal{A}_k} \frac{|g(x)-g(y)|^p |x|^{\alpha_{1, 1} p }  |y|^{\alpha_{1, 2} p}}{|x-y|^{d+sp}} \, dx \, dy}^{\gamma'\beta_1}  \\[6pt] 
 \times  \bct{\int_{\mcal{A}_k}\int_{\mcal{A}_k} \frac{|g(x)|^q|g(y)|^q |x|^{\alpha_{2, 1} q} |y|^{\alpha_{2, 2} q}}{|x-y|^{d-\alpha}} \, dx \, dy}^{\gamma'\beta_2}. 
\end{multline}

Let $m, \, n\in \mathbb{Z}$ be such that $m\leq n-2$. Summing \eqref{sameAnnEst} with respect to $k$ from $m$ to $n$,  we get
\begin{multline}\label{thm2-betaEquations-m}
\int_{\Sb{2^m<|x|<2^{n+1}}} |g|^{\gamma'} |x|^{\tau'\gamma'}   \leq  C\sum_{k=m}^n 2^{(\gamma'\tau'+d)k}\md{\fint_{\mcal{A}_k}g}^{\gamma'} \\[6pt]
+ C\sum_{k=m}^n \bct{\int_{\mcal{A}_k}\int_{\mcal{A}_k} \frac{|g(x)-g(y)|^p |x|^{\alpha_{1, 1} p }  |y|^{\alpha_{1, 2} p} }{|x-y|^{d+sp}}\, dx \, dy}^{\gamma'\beta_1} \\[6pt]
\times \bct{\int_{\mcal{A}_k}\int_{\mcal{A}_k} \frac{|g(x)|^q|g(y)|^q  |x|^{\alpha_{2, 1} q} |y|^{\alpha_{2, 2} q} }{|x-y|^{d-\alpha}}\, dx \, dy}^{\gamma'\beta_2}. 
\end{multline}
Applying \Cref{lem-Ineq}, we derive from \eqref{thm2-betaEquations-m} that  
\begin{multline}\label{thm2-p2}
\int_{\Sb{2^m<|x|<2^{n+1}}} |g|^{\gamma'} |x|^{\tau'\gamma'}     \leq  C\sum_{k=m}^n 2^{(\gamma'\tau'+d)k}\md{\fint_{\mcal{A}_k}g}^{\gamma'} \\[6pt]+ C  \left( \int_{\mR^d} \int_{\mR^d} \frac{|g(x) - g(y)|^p |x|^{\alpha_{1,1}p } |y|^{\alpha_{1, 2} p } }{|x - y|^{d + sp }}  \, dx \, dy \right)^{\gamma'\beta_1}  \\[6pt]
\times \left( \int_{\mR^d}\int_{\mR^d} \frac{|g(x)|^q|g(y)|^q  |x|^{\alpha_{2, 1} q} |y|^{\alpha_{2, 2} q} }{|x-y|^{d-\alpha}}\, dx \, dy \right)^{\gamma'\beta_2 }.  
\end{multline}

We next estimate the first term of the RHS of \eqref{thm2-p2}. We have, as in \eqref{thm2-p1},
\begin{multline}\label{thm2-p3}
\md{\fint_{\mcal{A}_k} g-\fint_{\mcal{A}_{k+1}}g}^{\gamma'} \leq   \frac{C}{2^{\frac{dk}{\gamma}}}  \bct{\int_{\mcal{A}_k\cup \mcal{A}_{k+1}}\int_{\mcal{A}_k\cup\mcal{A}_{k+1}} \frac{|g(x)-g(y)|^p}{|x-y|^{d+sp}} \, dx \, dy}^{\gamma'\beta_1} \\[6pt] \times \bct{\int_{\mcal{A}_k\cup \mcal{A}_{k+1}}\int_{\mcal{A}_k\cup\mcal{A}_{k+1}} \frac{|g(x)|^q|g(y)|^q }{|x-y|^{d-\alpha}} \, dx \, dy}^{\gamma'\beta_2}. 
\end{multline}
 With $c: =2/(1+2^{\gamma'\tau'+d})<1$ (since $\gamma' \tau' + d > 0$), we have $c 2^{\tau' \gamma'  + d } > 1$. We derive from \eqref{thm2-p3} that 
\begin{multline}\label{thm2-coucou}
2^{(\gamma'\tau'+d)k} \md{\fint_{\mcal{A}_k} g}^{\gamma'} \leq c 2^{(\gamma'\tau'+d)(k+1)} \md{\fint_{\mcal{A}_{k+1}}g} ^{\gamma'} \\[6pt]
+C \bct{\int_{\mcal{A}_k\cup\mcal{A}_{k+1}}\int_{\mcal{A}_k\cup\mcal{A}_{k+1}} \frac{|g(x)-g(y)|^p |x|^{\alpha_{1, 1} p }  |y|^{\alpha_{1, 2} p} }{|x-y|^{d+sp}}dxdy}^{\gamma'\beta_1} \\[6pt] \times \bct{\int_{\mcal{A}_k\cup\mcal{A}_{k+1}}\int_{\mcal{A}_k\cup\mcal{A}_{k+1}} \frac{|g(x)|^q |g(y)|^q |x|^{\alpha_{2, 1} q} |y|^{\alpha_{2, 2} q}}{|x-y|^{d-\alpha}}dxdy}^{\gamma'\beta_2}.
\end{multline}
Since $g$ has a compact support, we derive that, for large $n$, 
\begin{multline}\label{sumAvgEst}
\sum_{k=m}^n 2^{(\gamma'\tau'+d)k} \md{\fint_{\mcal{A}_k} g}^{\gamma'} \leq C \sum_{k=m}^n \bct{\int_{\mcal{A}_k\cup\mcal{A}_{k+1}}\int_{\mcal{A}_k\cup\mcal{A}_{k+1}} \frac{|g(x)-g(y)|^p |x|^{\alpha_{1, 1} p }  |y|^{\alpha_{1, 2} p} }{|x-y|^{d+sp}}dxdy}^{\gamma'\beta_1} \\[6pt] \times \bct{\int_{\mcal{A}_k\cup\mcal{A}_{k+1}}\int_{\mcal{A}_k\cup\mcal{A}_{k+1}} \frac{|g(x)|^q |g(y)|^q |x|^{\alpha_{2, 1} q} |y|^{\alpha_{2, 2} q}}{|x-y|^{d-\alpha}}dxdy}^{\gamma'\beta_2}. 
\end{multline}
Applying \Cref{lem-Ineq} and letting $m \to - \infty$, we obtain 
\begin{multline}\label{thm2-p4}
\sum_{k \in \mZ} 2^{(\gamma'\tau'+d)k} \md{\fint_{\mcal{A}_k} g}^{\gamma'} \le 
C \left( \int_{\mR^d} \int_{\mR^d} \frac{|g(x) - g(y)|^p |x|^{\alpha_{1,1}p } |y|^{\alpha_{1, 2} p } }{|x - y|^{d + sp }}  \, dx \, dy \right)^{\gamma'\beta_1} \\[6pt] 
\times \left( \int_{\mR^d}\int_{\mR^d} \frac{|g(x)|^q|g(y)|^q  |x|^{\alpha_{2, 1} q} |y|^{\alpha_{2, 2} q} }{|x-y|^{d-\alpha}}\, dx \, dy \right)^{\gamma'\beta_2}.  
\end{multline}

Combining \eqref{thm2-p2} and \eqref{thm2-p4} and letting $n \to + \infty$, $m \to - \infty$, we obtain $(i)$ of \Cref{thm2}. The proof of Step 1 is complete.

\medskip 
\noindent \textit{Step 2:}  We establish $(i)$ of  Theorem \ref{thm2} assuming   \eqref{thm2-cd4-2} and $\gamma' < \gamma$.

Since $\frac{1}{p} (sp - d - \alpha_1 p)  + \frac{1}{2q}(\alpha + d + \alpha_2 q)  \neq 0$, by scaling, without loss of generality, one might assume that 
\begin{equation}\label{thm2-Step-p1}
\int_{\mR^d} \int_{\mR^d} \frac{|g(x) - g(y)|^p |x|^{\alpha_{1,1}p } |y|^{\alpha_{1, 2} p } }{|x - y|^{d + sp }}  \, dx \, dy  = \int_{\mR^d}\int_{\mR^d} \frac{|g(x)|^q|g(y)|^q  |x|^{\alpha_{2, 1} q} |y|^{\alpha_{2, 2} q} }{|x-y|^{d-\alpha}}\, dx \, dy = 1. 
\end{equation}
It then suffices  to prove that
\begin{equation}\label{thm2-Step-p0}
\| |\cdot|^{\tau'} g\|_{L^{\gamma'}(\mR^d)} \le C. 
\end{equation}

Let $\beta_{1, 1}, \beta_{1, 2}$ and $\beta_{2, 1}, \beta_{2, 2}$ close to $\beta_1$ and $\beta_2$ respectively and non-negative be determined later such that $\beta_{j, 1} p + 2 \beta_{j, 2} q = 1$, which is equivalent to 
$$
p(\beta_{j, 1} - \beta_1) + 2q (\beta_{j, 2} - \beta_2)  = 0. 
$$
 
Define $\sigma_j$, $\gamma_j$, $\gamma_j'$, and $\tau_j$  for $j = 1, 2$,  as follows 
$$
\sigma_j = \beta_{j, 1} p \alpha_1 + \beta_{j, 2} q \alpha_2, 
$$
$$
(d- sp) \beta_{j, 1} + (d+ \alpha) \beta_{j, 2} = \frac{d}{\gamma_j}. 
$$
$$
\gamma_j' = \gamma_j, \quad \tau_j' = \sigma_j.  
$$

We have 
\begin{multline}
\frac{1}{\gamma_j'} + \frac{\tau_j'}{d} - \frac{1}{\gamma'} - \frac{\tau'}{d} = \frac{1}{\gamma_j}  + \frac{\sigma_j}{d} - \frac{1}{\gamma} - \frac{\sigma}{d} \\[6pt]=
\frac{1}{d} \Big( (d- sp + p \alpha_1) (\beta_{j, 1}- \beta_1) + (d+ \alpha + q \alpha_2) (\beta_{j, 2} - \beta_2) \Big). 
\end{multline}

Since $2q(d- sp + p \alpha_1) \neq p (d+ \alpha + q \alpha_2)$  and $\beta_1, \beta_2 > 0$,  one can choose positive $\beta_{j, 1}$ and $\beta_{j, 2}$ close to $\beta_1$ and $\beta_2$ such that 
\begin{equation}\label{thm2-Step2-p1*}
p(\beta_{j, 1} - \beta_1) + 2q (\beta_{j, 2} - \beta_2)  = 0, 
\end{equation}
\begin{equation}\label{thm2-Step2-p2*}
\frac{1}{\gamma_2'} + \frac{\tau_2'}{d} < \frac{1}{\gamma'} + \frac{\tau'}{d}<\frac{1}{\gamma_1'} + \frac{\tau_1'}{d}. 
\end{equation}

Since $\beta_{j, 1}$ and $\beta_{j, 2}$ are close to $\beta_1$ and $\beta_2$ and $\beta_1 \gamma + \beta_2 \gamma  > 1$, we have 
\begin{equation}\label{thm2-Step2-p3*}
\beta_{j, 1} \gamma_j' + \beta_{j, 2} \gamma_j'  = \beta_{j, 1} \gamma_j + \beta_{j, 2} \gamma_j  > 1 \mbox{ for } j =1, 2,  
\end{equation}
and 
\begin{equation}\label{thm2-Step2-p4*}
\gamma_j' = \gamma_j > \gamma' \mbox{ for } j = 1, 2. 
\end{equation}

Combining \eqref{thm2-Step2-p2*} and \eqref{thm2-Step2-p4*} yields 
\begin{equation}\label{thm2-Step-p2}
\| |\cdot|^{\tau'} g\|_{L^{\gamma'}(\mR^d \setminus B_1)} \le C \| |\cdot|^{\tau_1'} g\|_{L^{\gamma_1'}(\mR^d \setminus B_1)}. 
\end{equation}
and
\begin{equation}\label{thm2-Step-p3}
\| |\cdot|^{\tau'} g\|_{L^{\gamma'}(B_1)} \le C \| |\cdot|^{\tau_2'} g\|_{L^{\gamma_2'}(B_1)}
\end{equation}
On the other hand, applying Step 1, we have, by \eqref{thm2-Step-p1}, 
\begin{equation}\label{thm2-Step-p4}
 \| |\cdot|^{\tau_1'} g\|_{L^{\gamma_1'}(\mR^d)} \le C 
 \quad \mbox{ and } \quad 
 \| |\cdot|^{\tau_2'} g\|_{L^{\gamma_2'}(\mR^d)} \le C. 
\end{equation}
Combining \eqref{thm2-Step-p1}, \eqref{thm2-Step-p2}, and \eqref{thm2-Step-p3} yields \eqref{thm2-Step-p0}. The proof of Step 2 is complete.

\medskip 
The proof of $(i)$ of \Cref{thm2} is complete. \qed

\subsection{Proof of $(ii)$ of \Cref{thm2}}  The proof of $(ii)$ of \Cref{thm2} is similar to that of (i). We only mention briefly the proof of $(ii)$ assuming \eqref{thm2-cd4-1}.  We also have \eqref{thm2-p2}. To estimate the RHS of \eqref{thm2-p2}, one just needs to note that, instead of \eqref{thm2-coucou}, 
with $\hat c:=(1+2^{\gamma'\tau'+d})/2<1$ (since $\gamma' \tau'  + d <  0$), 
\begin{multline*}
2^{(\gamma'\tau'+d)(k+1)} \md{\fint_{\mcal{A}_{k+1}} g}^{\gamma'}   \leq \hat c 2^{(\gamma'\tau'+d)k} \md{\fint_{\mcal{A}_{k}}g} ^{\gamma'} \\[6pt]
+C \bct{\int_{\mcal{A}_k\cup\mcal{A}_{k+1}}\int_{\mcal{A}_k\cup\mcal{A}_{k+1}} \frac{|g(x)-g(y)|^p \varphi_{\alpha_1, p}(x, y) }{|x-y|^{d+sp}}dxdy}^{\gamma'\beta_1} \\[6pt] \times \bct{\int_{\mcal{A}_k\cup\mcal{A}_{k+1}}\int_{\mcal{A}_k\cup\mcal{A}_{k+1}} \frac{|g(x)|^q|g(y)|^q \varphi_{\alpha_2, q}(x, y)}{|x-y|^{d-\alpha}}dxdy}^{\gamma'\beta_2}. \end{multline*}
Summing with respect to $k$, we also obtain \eqref{thm2-p4}. The conclusion now follows from \eqref{thm2-p2} and \eqref{thm2-p4}.

\section{Hardy-Lieb-Therring Inequalities} \label{sect-P}

In this section, as an application of \Cref{thm2}, we establish the following family of many body Hardy-Lieb-Therring inequalities. 

\begin{thm}\label{thm4}
Let $d \ge 1$ and $N \ge 1$, $0<s \le 1$ and  $2\leq p<\infty$ be such that $sp<d$. Given $\psi \in C_c^\infty\bct{\ro^{dN}}$ with $\int_{\ro^{dN}} |\psi(X)|^p \, dX =1$, define 
\begin{equation}\label{density}
\rho\psi(x):=\sum_{i=1}^N \int_{(\ro^d)^{N-1}} \md{\psi(\Xr_i,x,\Xl_i)}^p d\Xr_i d\Xl_i,
\end{equation}
 where we have denoted, for $X = (x_1, \cdots, x_N)$ with $x_i \in \mR^d$,  
 \begin{equation}\label{thm4-notation}
\mbox{ $\Xr_i= (x_1,\dots,x_{i-1})$ and $\Xl_i=(x_{i+1},\dots,x_N)$ for $1 \le i \le N$.  } 
\end{equation}
Then there exists a positive constant $C$ depending only on $s$, $p$, $d$ ($C$ is independent of $N$) such that
\begin{equation}
E (\psi) +\sum_{1\leq i<j\leq N} \int_{\ro^{dN}} \frac{|\psi(X)|^p dX}{|x_i-x_j|^{sp}} \geq C \int_{\ro^d} \md{\rho\psi (x)}^{1+\frac{sp}{d}} \, dx. \label{thm4-ineq}
\end{equation}
Here, for $0< s< 1$,  
\begin{multline*}
E (\psi): = \sum_{i=1}^N  \int_{\ro^{d(N-1)}} \left[\int_{\ro^d} \int_{\ro^d}\frac{\md{\psi(\Xr_i,x_i,\Xl_i) -  \psi(\Xr_i,y_i,\Xl_i)}^p dx_idy_i}{\md{x_i-y_i}^{d+sp}} \right. \\[6pt] \left.- \mcal{C}_{d,s,p}\int_{\ro^d} \frac{\md{\psi(\Xr_i,x_i,\Xl_i)}^pdx_i}{\md{x_i}^{sp}} \right]d\Xr_id\Xl_i, 
\end{multline*}
and, for $s=1$,  
\begin{equation*}
E (\psi): = \sum_{i=1}^N \int_{\ro^{d(N-1)}} \left[\int_{\mR^d} |\nabla_{x_i} \psi(X_i^R, x_i, X_i^L)|^p \, dx_i  - \mcal{C}_{d,s,p}\int_{\ro^d} \frac{\md{\psi(\Xr_i,x_i,\Xl_i)}^pdx_i}{\md{x_i}^{sp}} \right] d\Xr_id\Xl_i.   
\end{equation*}
\end{thm}

The constant  $\mcal{C}_{d,s,p}$  in \Cref{thm4} is defined by, for $0 < s < 1$, 
\begin{equation}\label{harOpt}
\mcal{C}_{d,s,p} := 2\int_0^1 r^{sp-1}\md{1-r^\frac{d-sp}{p}}^p \Phi_{d,s,p}(r) \, dr, 
\end{equation}
where 
\begin{equation}
\Phi_{d,s,p} = 
\left\{\begin{array}{cl}
\dsp \md{\mathbb{S}^{d-2}} \int_{-1}^1  \frac{\bct{1-t^2}^\frac{d-3}{2} dt}{\bct{1-2rt+r^2}^\frac{d+sp}{2}} & \mbox{ for }  d\geq2, \\[6pt] \dsp \frac{1}{(1-r)^{1+ps}}+ \frac{1}{(1+r)^{1+ps}} & \mbox{ for } d=1,  
\end{array}\right. 
\end{equation}
and 
$$
\mcal{C}_{d,1,p} = \left(\frac{n-p}{p}\right)^p. 
$$ 

When $s=1$, $p=2$, the following many body inequality was derived by Lieb and Thirring \cite{LT75, LT76} in order to give a simpler proof of the stability of non-relativistic matter first given by Dyson and Lenard \cite{LD68}.
\begin{equation}\label{LTIn}
\langle \psi, \sum_{i=1}^N -\Delta_i\psi \rangle \geq C_{LT} \int_{\ro^d} \rho\psi ^\frac{d+2}{d},
\end{equation}
where $\psi \in H^1(\mR^{dN})$, anti-symmetric and normalized in $L^2(\mR^{dN})$. $\rho\psi$ is same as in \eqref{density}, and $C_{LT}>0$ is a constant  independent of $\psi$ and $N$. It is easy to that, \eqref{LTIn} is no longer true if $\psi$ is not anti-symmetric e.g. if $\psi(x_1,\dots,x_N)= u(x_1)\dots u(x_N)$, which is a typical state of boson. In \cite{LunPorSol}, Lundholm, Portmann and Solovej  noticed that Lieb-Thirring type inequalities still hold true for particles without any symmetry assumptions and therefore in particular for bosons, provided that the anti-symmetry assumption is replaced by a sufficiently strong repulsive interaction between particles. More precisely, they established \eqref{thm4-ineq} for $s=1$, $p=2$ in the absence of the inverse square potential $\frac{1}{|x_i|^2}$. Subsequently, Lundholm, Nam and Portmann \cite{LunNamPor} established an improved version of this inequality. In fact, they proved \eqref{thm4-ineq} for $s>0$ and $p=2$. Our approach of proving \Cref{thm4} is as follows. First using \Cref{thm2} and a (fractional) Hardy inequality due to Frank and Seiringer \cite{FrSei}, we derive \Cref{harCs} below. Then we follow the strategy of \cite{LunNamPor} to derive \Cref{thm4} from \Cref{harCs}.

\medskip 

To prove \Cref{thm4}, we will establish the following  Hardy-Lieb-Therring inequality.

\begin{pro}\label{harCs}
Let $d\geq1$,  $0<s \le 1$ and $p\geq2$ be such that $sp<d$. Then there exist $C>0$ such that for any $u \in C_c^\infty(\ro^d)$ we have 

\begin{equation}\label{harCs1}
F(u)^{1-\frac{sp}{d}} \\[6pt]
\times  \bct{\int_{\ro^d}\int_{\ro^d} \frac{|u(x)|^p|u(y)|^p dxdy}{|x-y|^{sp}}}^{\frac{sp}{d}}  \geq  C \int_{\ro^d} \md{u}^{\frac{p(d+sp)}{d}},
\end{equation}
where 
\begin{equation}\label{harCs1}
F(u): = \left\{ \begin{array}{cl} \dsp \bct{\int_{\ro^d}\int_{\ro^d} \frac{\md{u(x) -u(y)}^p dx dy}{\md{x-y}^{d+sp}}- \mcal{C}_{d,s,p} \int_{\ro^d} \frac{|u(x)|^pdx}{|x|^{sp}}} & \mbox{ for } 0 < s <1,  \\[6pt]
\dsp \int_{\mR^d} |\nabla u (x)|^p \, dx - \mcal{C}_{d,1,p} \int_{\ro^d} \frac{|u(x)|^pdx}{|x|^{sp}} & \mbox{ for } s = 1. 
\end{array}\right. 
\end{equation}
\end{pro}

The proof of \Cref{harCs} contains two main ingredients. The first one is \Cref{thm2} and the second one is a (fractional) Hardy inequality due to Frank and Seiringer \cite{FrSei}.

\medskip 
The rest of this section containing two subsection is organized as follows. We prove \Cref{harCs} and \Cref{thm4}  in the first subsection and the second subsection, respectively.

\subsection{Proof of \Cref{harCs}}
Let $u\in C_c^{\infty}(\ro^d)$ and define $\varphi (x):=|x|^{(d-sp)/p}u(x)$. We first consider the case $0< s < 1$. 
We have, see \cite[Theorem 1.2]{FrSei},   
\begin{multline}\label{harCs-p1}
\iint_{\ro^d\times\ro^d} \frac{|u(x)-u(y)|^p}{|x-y|^{N+sp}} dx dy -\mcal{C}_{d,s,p} \int_{\ro^d} \frac{|u|^p}{|x|^{sp}} dx  \\[6pt]
\geq c_p \iint_{\ro^d\times \ro^d} \frac{\Big| \varphi (x) - \varphi(y) \Big|^p |x|^{-\frac{d-sp}{2}}|y|^{-\frac{d-sp}{2}}}{|x-y|^{d+sp}} dx dy.
\end{multline}
where  $c_p:= \min_{0<r<\frac{1}{2}}\bct{(1-r)^p-r^p+pr^{p-1} }$.

We apply Theorem \ref{thm2} to $\varphi$ with 
\begin{equation*}
p=q\geq 2, \quad \gamma' = (d+sp)\frac{p}{d}, \quad  \alpha=d-sp,  \quad \beta_1 = \frac{d-sp}{p(d+sp)}, \quad \beta_2 =  \frac{s}{d+sp}, 
\end{equation*} 
\begin{equation*}
\tau'= -\frac{d-sp}{p}, \quad \alpha_{1,1}= \alpha_{1,2} = -\frac{d-sp}{2p},\quad \mbox{ and } \quad \alpha_{2,1}=  \alpha_{2,2} = -\frac{d-sp}{p}.
\end{equation*}
We then have 
\begin{equation*}
\alpha_1= \alpha_{1,1}+\alpha_{1,2} = -\frac{d-sp}{p}, \quad  \alpha_2=\alpha_{2,1}+\alpha_{2,2}  = -\frac{2(d-sp)}{p}, 
\end{equation*}
\begin{equation*}
\sigma= \beta_1p\alpha_1+\beta_2q\alpha_2= -(d-sp)/p =\tau', 
\end{equation*} 
\begin{equation*}
\frac{1}{\gamma}= \frac{1}{d}[(d-sp)\beta_1+(d+\alpha) \beta_2]= \frac{d}{p(d+sp)}s = \frac{1}{\gamma'}.
\end{equation*} 
One can check that  \eqref{betaEquations-m} holds and $\frac{1}{\gamma'}+\frac{\tau'}{d} = \frac{s^2p}{d(d+sp)}>0$. 

We then obtain 
\begin{multline}\label{harCs-p2}
\bct{\int_{\ro^d}\int_{\ro^d} \frac{\md{ \varphi (x) - \varphi (y)}^p |x|^{-\frac{d-sp}{2}} |y|^{-\frac{d-sp}{2}} dx dy}{\md{x-y}^{d+sp}}}^{1-\frac{sp}{d}} \\[6pt]
\times \bct{\int_{\ro^d}\int_{\ro^d} \frac{| \varphi (x)|^p| \varphi (y)|^p |x|^{-(d-sp)} |y|^{-(d-sp)} dxdy}{|x-y|^{sp}}}^{\frac{sp}{d}} \geq  C \int_{\ro^d} \md{|x|^{-\frac{d-sp}{p}}\varphi}^{\frac{p(d+sp)}{d}} ,
\end{multline}
Putting $\varphi = |x|^\frac{d-sp}{p}u$ in  \eqref{harCs-p2} and then using \eqref{harCs-p1},  we derive \eqref{harCs1}. 

The proof in the case $0 < s < 1$ is complete.

The proof in the case $s=1$ follows similarly. In this case,  instead of \eqref{harCs-p1}, one has, see \cite[Remark 2.5]{FrSei} (see also \cite[Theorem 1]{IIO}), 
\begin{equation}\label{harCs-p1*}
\int_{\mR^d} |\nabla u(x)|^p \, dx - \mcal{C}_{d,1,p} \int_{\ro^d} \frac{|u(x)|^pdx}{|x|^{sp}} \ge  c_p \int_{\mR^d}  \frac{| \varphi (x)|^p |x|^{-(d-p)}}{|x|^{p}} dx.
\end{equation}
The rest of the proof is almost unchanged and is omitted. \qed

\subsection{Proof of \Cref{thm4}}

The proof of \Cref{thm4} is based on \Cref{harCs} and the following three lemmas. The first one is

\begin{lem} \label{hoIn}
Let $d\ge 1$, $N \ge 1$, $0<s<1$, $1 \le p < + \infty$, $\alpha_1,\alpha_2\in \ro$ and  $\psi$ be a measurable function defined  in $\ro^{dN}$. We have  
\begin{multline}\label{hoIneq}
\int_{\ro^d}\int_{\ro^d} \frac{\md{\bct{\rho\psi}^\frac{1}{p}(x) - \bct{\rho\psi}^\frac{1}{p}(y) }^p |x|^{\alpha_1p}|y|^{\alpha_2p}}{\md{x-y}^{d+sp}}dx dy \\[6pt]
\leq \sum_{i=1}^N \int_{\ro^{d(N-1)}}\int_{\ro^d}\int_{\ro^d} \frac{\md{\psi(\Xr_i,x_i,\Xl_i) - \psi(\Xr_i,y_i,\Xl_i) }^p|x_i|^{\alpha_1p}|y_i|^{\alpha_2p}} {\md{x_i-y_i}^{d+sp}}dx_idy_id\Xr_id\Xl_i. 
\end{multline}
Here $\rho\psi$ is defined by \eqref{density}, and $X_i^R$ and $X_i^L$ are given by \eqref{thm4-notation}. 
\end{lem}

For $s=1$, we use the following lemma.
\begin{lem} \label{hoIn-s1}
Let $d\ge 1$, $N \ge 1$, $1 \le p  < + \infty$, $\alpha_1\in \ro$ and  $\psi$ be a smooth function defined  in $\ro^{dN}$. We have  
\begin{equation}\label{hoIneq-s1}
\int_{\ro^d}\md{\nabla\bct{\rho\psi}^\frac{1}{p}(x)}^p |x|^{\alpha_1p}dx\leq\sum_{i=1}^N \int_{\ro^{dN}} \md{\nabla_{x_i}\psi(X)}^p|x_i|^{\alpha_1p}dX. 
\end{equation}
Here $\rho\psi$ is defined by \eqref{density}. 
\end{lem}

Inequality \eqref{hoIneq-s1} was first discovered by Hoffman--Ostenhof  in \cite{HofOst},  $p=2$  and $\alpha_1=\alpha_2=0$. Later on, in \cite[Lemma 8.4]{LS10}, \eqref{hoIneq},\eqref{hoIneq-s1} was proved  for $p=2$ and $\alpha_1=\alpha_2=0$.  The proofs in the general cases stated here follows in the same spirit. For the convenience of the proof, we give the proof  in \Cref{sect-holn}.

\medskip 
The third lemma used in the proof of \Cref{thm4} is

\begin{lem} \label{liOxIn}  Let $d \ge 1$ and $1 \le p < + \infty$. 
For every $0<\gamma<d$ and for every $\psi\in L^p\bct{\ro^{dN}}$ with $\dsp \int |\psi (X)|^p \, dX= 1$, we have
\begin{align}\label{liOxIneq}
\int_{\mR^{dN}} \sum_{1\leq i<j \leq N} \frac{|\psi (X)|^p}{\md{x_i-x_j}^\gamma} \, dX \geq \frac{1}{2}\int_{\ro^d}\int_{\ro^d} \frac{\rho\psi(x)\rho\psi(y)}{|x-y|^\gamma} dx dy -C_{LO} \int_{\ro^d} \md{\rho\psi (x)}^{1+\frac{\gamma}{d}}  \, dx 
\end{align}
for a constant $C_{LO}>0$ depending only on $d$, $\gamma$ and $p$.
\end{lem}

Inequality \eqref{liOxIneq} was first studied in \cite{Li'79,LiOx} for the case $\gamma=1$, $d=3$ and $p=2$ and known under the name 
the Lieb-Oxford inequality for homogeneous potential. Subsequently, it was derived in \cite[Lemma 5.3]{LiSolYng} for $\gamma=1$, $d=2$, $p=2$ and in \cite[Lemma 16]{LunNamPor} for the case $0<\gamma<d$ and $p=2$. Interestingly, the proof of \Cref{liOxIn} in the assumption stated there  does not differ much from that of \cite{LiSolYng} or \cite{LunNamPor}  in the case $p=2$,  and is given in \Cref{sect-liOxIn} for the completeness.

\medskip 
We are ready to give 

\medskip 
\noindent {\bf Proof of Theorem \ref{thm4}.} We only consider the case $0 < s < 1$, the proof in general case follows similarly and is omitted.  We have
\begin{equation}\label{thm4-p0}
\sum_{i=1}^N \int_{\ro^{dN}} \frac{\md{\psi(X)}^p \, dX}{|x_i|^{sp}} = \int_{\ro^d} \frac{\rho\psi(x)}{|x|^{sp}} dx.
\end{equation}
Applying \Cref{hoIn} and using \eqref{thm4-p0}, we obtain 
\begin{multline}\label{thm4-p1}
\sum_{i=1}^N  \int_{\ro^{d(N-1)}} \left[\int_{\ro^d}\int_{\ro^d} \frac{\md{\psi(\Xr_i,x_i,\Xl_i) -  \psi(\Xr_i,y_i,\Xl_i)}^p dx_idy_i}{\md{x_i-y_i}^{d+sp}} 
\right. \\[6pt] \left.- \mcal{C}_{d,s,p}\int_{\ro^d} \frac{\md{\psi(\Xr_i,x_i,\Xl_i)}^pdx_i}{\md{x_i}^{sp}} \right]  d\Xr_id\Xl_i  \\[6pt]
\geq \int_{\ro^d}\int_{\ro^d} \frac{\md{\bct{\rho\psi}^\frac{1}{p}(x) - \bct{\rho\psi}^\frac{1}{p}(y)}^p dx dy}{\md{x-y}^{d+sp}}  - \mcal{C}_{d,s,p}\int_{\ro^d} \frac{\rho\psi(x) \, dx}{|x|^{sp}}. 
\end{multline}
By \eqref{liOxIneq},  we obtain,  for $0<\eps<1$,
\begin{equation}\label{thm4-p2}
\eps \sum_{1\leq i<j\leq N} \int_{(\ro^d)^N} \frac{|\psi(X)|^p dX}{|x_i-x_j|^{sp}} \geq  \frac{\eps}{2} \int_{\ro^d}\int_{\ro^d} \frac{\rho\psi(x)\rho\psi(y) dx dy}{|x-y|^{sp}} - C_{LO} \eps \int_{\ro^d} |\rho\psi (x)|^{1+\frac{sp}{d}} \, dx. 
\end{equation}
Combining \eqref{thm4-p1} and \eqref{thm4-p2} yields 
\begin{multline}\label{thm4-p12}
\sum_{i=1}^N \int_{(\ro^d)^{N-1}} \left[\int_{\ro^d}\int_{\ro^d} \frac{\md{\psi(\Xr_i,x_i,\Xl_i) -  \psi(\Xr_i,y_i,\Xl_i)}^p dx_idy_i}{\md{x_i-y_i}^{d+sp}} 
\right. \\[6pt] \left.- \mcal{C}_{d,s,p}\int_{\ro^d} \frac{\md{\psi(\Xr_i,x_i,\Xl_i)}^pdx_i}{\md{x_i}^{sp}} \right]  d\Xr_id\Xl_i  + \eps \sum_{1\leq i<j\leq N} \int_{(\ro^d)^N} \frac{|\psi(X)|^p dX}{|x_i-x_j|^{sp}} \\[6pt]
\geq \int_{\ro^d}\int_{\ro^d} \frac{\md{\bct{\rho\psi}^\frac{1}{p}(x) - \bct{\rho\psi}^\frac{1}{p}(y)}^p dx dy}{\md{x-y}^{d+sp}}  - \mcal{C}_{d,s,p}\int_{\ro^d} \frac{\rho\psi(x)dx}{|x|^{sp}} \\[6pt]+ \frac{\eps}{2} \int_{\ro^d}\int_{\ro^d} \frac{\rho\psi(x)\rho\psi(y) dx dy}{|x-y|^{sp}} - C_{LO} \eps \int_{\ro^d} |\rho\psi(x)|^{1+\frac{sp}{d}}. 
\end{multline}

Since, by Young's inequality after noting that $sp/d < 1$,  
\begin{multline*}
 \left[\int_{\ro^d}\int_{\ro^d} \frac{\md{\bct{\rho\psi}^\frac{1}{p}(x) - \bct{\rho\psi}^\frac{1}{p}(y)}^p dx dy}{\md{x-y}^{d+sp}}  -\mcal{C}_{d,s,p} \int_{\ro^d} \frac{\rho\psi(x)dx}{|x|^{sp}}\right] +\epsilon  \int_{\ro^d}\int_{\ro^d} \frac{\rho\psi(x)\rho\psi(y) dx dy }{|x-y|^{sp}} \\[6pt] \geq C \epsilon^\frac{sp}{d}  \bct{\int_{\ro^d}\int_{\ro^d} \frac{\md{\bct{\rho\psi}^\frac{1}{p}(x) - \bct{\rho\psi}^\frac{1}{p}(y)}^p dx dy}{\md{x-y}^{d+sp}}  - \mcal{C}_{d,s,p}\int_{\ro^d} \frac{\rho\psi(x)dx}{|x|^{sp}}}^{1-\frac{sp}{d}}   \\[6pt]\times \bct{\int_{\ro^d}\int_{\ro^d} \frac{\rho\psi(x)\rho\psi(y) dx dy }{|x-y|^{sp}}}^\frac{sp}{d}, 
\end{multline*}
applying  \Cref{harCs} to $(\rho\psi)^{1/p}$,  we obtain 
 \begin{multline}\label{thm4-p3} \left[\int_{\ro^d}\int_{\ro^d} \frac{\md{\bct{\rho\psi}^\frac{1}{p}(x) - \bct{\rho\psi}^\frac{1}{p}(y)}^p dx dy}{\md{x-y}^{d+sp}}  -\mcal{C}_{d,s,p} \int_{\ro^d} \frac{\rho\psi(x)dx}{|x|^{sp}}\right] \\[6pt] 
 +\epsilon \int_{\ro^d}\int_{\ro^d} \frac{\rho\psi(x)\rho\psi(y) dx dy }{|x-y|^{sp}} 
\geq C\epsilon^\frac{sp}{d} \int_{\ro^d} (\rho\psi)^{1+sp/d}.
\end{multline}

By choosing $\eps$ sufficiently so that  $C\epsilon^\frac{sp}{d}-\epsilon C_{LO}>0$ (this can be done since $sp < d$), we derive from \eqref{thm4-p12} and \eqref{thm4-p3} the conclusion. The proof is complete. 
 \qed 

\appendix

\section{Proof of Lemma \ref{hoIn} and \ref{hoIn-s1}} \label{sect-holn}
First we consider $s=1$,  i.e.,  we establish \Cref{hoIn-s1}. In this case \eqref{hoIneq-s1} is a direct consequence of the following estimate which is proved by using  H\"older's inequality with exponents $p$, $\frac{p}{p-1}$. 
$$\md{\nabla_x \rho\psi(x)}^p \leq \bct{\rho\psi (x)}^{p-1}\sum _{i=1}^N\int_{\mR^{d(N-1)}} \md{\nabla _x \psi (X_i^R, x, X_i^L)}^p dX_i^RdX_i^L.$$

We next consider $0<s<1$,  i.e.,  we prove \Cref{hoIn}. We have, by \eqref{density}, 
\begin{multline}\label{holn-p1}
\md{\bct{\rho\psi}^\frac{1}{p}(x) - \bct{\rho\psi}^\frac{1}{p}(y) }^p = \md{ \bct{\sum_{i=1}^N \int_{\ro^{d(N-1)}} \md{\psi(\Xr_i,x,\Xl_i)}^pd\Xr_id\Xl_i }^\frac{1}{p}  \right.\\[6pt] \left. - \bct{\sum_{i=1}^N \int_{\ro^{d(N-1)}} \md{\psi(\Xr_i,y,\Xl_i)}^pd\Xr_id\Xl_i}^\frac{1}{p} }^p. 
\end{multline}
By Minkowski's inequality,  
\begin{equation}\label{holn-p2}
\mbox{the RHS of \eqref{holn-p1}} \le   \sum_{i=1}^N \int_{\ro^{d(N-1)}} \md{\psi(\Xr_i,x,\Xl_i)- \psi(\Xr_i,y,\Xl_i)}^p d\Xr_id\Xl_i.
\end{equation}
Combining \eqref{holn-p1} and \eqref{holn-p2} yields 
\begin{align*}
&\int_{\ro^d}\int_{\ro^d} \frac{\md{\bct{\rho\psi}^\frac{1}{p}(x) - \bct{\rho\psi}^\frac{1}{p}(y) }^p |x|^{\alpha_1p}|y|^{\alpha_2p}}{\md{x-y}^{d+sp}}dx dy \\ &\leq \sum_{i=1}^N \int_{\ro^{d(N-1)}} \int_{\ro^d}\int_{\ro^d} \frac{\md{\psi(\Xr_i,x,\Xl_i)- \psi(\Xr_i,y,\Xl_i)}^p |x|^{\alpha_1p}|y|^{\alpha_2p}}{\md{x-y}^{d+sp}} dxdyd\Xr_id\Xl_i,
\end{align*}
which is the conclusion. \qed

\section{Proof of Lemma \ref{liOxIn}} \label{sect-liOxIn}

Let $\chi_R$ denote the characteristic function of the ball $\overline{B(0,R)}$. Then the following identity,  due to Fefferman-de la Llave \cite{FefDll}  (see also, \cite[Theorem 9.8]{LiebLoss}) holds
\begin{equation}\label{liOxln-p1}
\frac{1}{|x-y|^\gamma} = c_{d,\gamma} \int_0^\infty \int_{\ro^d} \chi_R(x-z)\chi_R(y-z) \frac{dzdR}{R^{d+\gamma+1}}.
\end{equation}
for some positive constant $c_{\gamma, d}$ depending only on $\gamma$ and $d$. Denote 
$$
f_R(z)= \int_{\ro^d} \rho\psi(x) \chi_R(x-z) dx  \mbox{ for } z \in \ro^d.
$$
It follows from \eqref{liOxln-p1} that 
\begin{align}\label{liOxln-p2}
\int_{\ro^d}\int_{\ro^d} \frac{\rho\psi(x)\rho\psi(y)}{|x-y|^\gamma} dx dy = c_{d,\gamma} \int_0^\infty\int_{\ro^d} f^2_R(z) \frac{dzdR}{R^{d+\gamma+1}},
\end{align}

Using \eqref{liOxln-p1}, we also have
\begin{multline}\label{liOxln-p3}
\int_{\ro^{dN}} |\psi (X)|^p \sum_{1\leq i<j\leq N} \frac{1}{|x_i-x_j|^\gamma} dX \\[6pt]
= c_{d,\gamma} \int_0^\infty \int_{\ro^d} \int_{\ro^{dN}} |\psi(X)|^p\sum_{1\leq i<j\leq N} \chi_R(x_i-z)\chi_R(x_j-z) \frac{dXdzdR}{R^{d+\gamma+1}} . 
\end{multline}

We have 
\begin{multline}\label{liOxln-p4}
\sum_{1\leq i<j\leq N} \chi_R(x_i-z)\chi_R(x_j-z) = \frac{1}{2} \bct{\sum_{i=1}^N \chi_R(x_i-z)}^2 -\frac{1}{2} \sum_{i=1}^N \chi_R(x_i-z) \\[6pt]= \frac{1}{2} \bct{\sum_{i=1}^N \chi_R(x_i-z)-f_R(z)}^2 + f_R(z)\sum_{i=1}^N \chi_R(x_i-z)-\frac{1}{2} f_R^2(z) -\frac{1}{2}\sum_{i=1}^N \chi_R(x_i-z).
\end{multline}
and 
\begin{align}\label{liOxln-p5}
\int_{\ro^{dN}} |\psi(X)|^p \sum_{i=1}^N \chi_R(x_i-z) dX = & \sum_{i=1}^N \int_{\ro^d}\int_{\ro^{d(N-1)}} \chi_R(x_i-z) |\psi(\Xr_i,x_i,\Xl_i)|^p \, dx_i \, d\Xr_i \, d\Xl_i \\[6pt]= & \int_{\ro^d} \chi_R(x-z)\rho\psi(x) dx = f_R(z) \nonumber
\end{align}
From  \eqref{liOxln-p4}  we obtain
\begin{multline}
\int_{\ro^{dN}} |\psi(X)|^p \sum_{1\leq i<j\leq N} \chi_R(x_i-z)\chi_R(x_j-z) dX \\[6pt]
\ge  \int_{\ro^{dN}} |\psi(X)|^p  \left( f_R(z)\sum_{i=1}^N \chi_R(x_i-z)-\frac{1}{2} f_R^2(z) -\frac{1}{2}\sum_{i=1}^N \chi_R(x_i-z)\right). 
\end{multline}
We then derive from \eqref{liOxln-p5} and the fact $\int_{\ro^{dN}}|\psi (X)|^p \, d X=1$ that 
\begin{equation}\label{liOxln-p6}
\int_{\ro^{dN}} |\psi(X)|^p \sum_{1\leq i<j\leq N} \chi_R(x_i-z)\chi_R(x_j-z) dX \ge \frac{1}{2}f_R^2(z)- \frac{1}{2}f_R(z).
\end{equation}

Plugging  \eqref{liOxln-p6} into \eqref{liOxln-p3}, using  
\eqref{liOxln-p2} and  the fact $f_R(z) \ge \min\{f_R(z), f_R^2(z)\} $, we obtain
\begin{multline}\label{liOxln-p7}
\int_{\ro^{dN}} |\psi (X)|^p \sum_{1\leq i<j\leq N} \frac{1}{|x_i-x_j|^\gamma} dX \\[6pt] \geq  \frac{1}{2} \int_{\ro^d}\int_{\ro^d} \frac{\rho\psi(x)\rho\psi(y)}{|x-y|^\gamma} dx dy -\frac{c_{d,\gamma}}{2}\int_0^\infty \int_{\ro^d} \min\{f_R(z), f_R^2(z)\} \frac{dzdR}{R^{d+\gamma+1}}. 
\end{multline}

Let $\rho^*$ denote  the Hardy-Littlewood maximal function of $\rho\psi$, i.e., 
\begin{equation}\label{liOxln-rho} 
\rho^*(z) := \sup_{R>0} \fint_{B(z,R)} \rho\psi(x) dx = \md{B_1}^{-1} \sup_{R>0} \frac{f_R(z)}{R^d}.
 \end{equation}

For $R^*>0$ and $z \in \mR^d$, we have
\begin{align*}
\int_0^\infty \min\{f_R^2(z),f_R(z)\} \frac{dR}{R^{d+\gamma+1}} &\leq \int_{0}^{R^*} f_R^2(z) \frac{dR}{R^{d+\gamma+1}} +\int_{R^*}^\infty f_{R}(z) \frac{dR}{R^{d+\gamma+1}} \\ & \mathop{\leq}^{\eqref{liOxln-rho}} \int_0^{R^*} \bct{\md{B_1}R^d \rho^*(z)}^2 \frac{dR}{R^{d+\gamma+1}} + \int_{R^*}^\infty \md{B_1} R^d \rho^*(z) \frac{dR}{R^{d+\gamma+1}} \\ &= \frac{\md{B_1}^2}{d-\gamma} (R^*)^{d-\gamma} \bct{\rho^*(z)}^2 + \frac{\md{B_1}}{\gamma} \bct{R^*}^{-\gamma} \rho^*(z),
\end{align*}
Taking $R^*= \bct{\md{B_1} \rho^*(z)}^{-\frac{1}{d}}$,  we derive that   
\begin{equation}\label{liOxln-p8}
\int_0^\infty \min\{f_R^2(z),f_R(z)\} \frac{dR}{R^{d+\gamma+1}} \leq \frac{d}{\gamma(d-\gamma)} \md{B_1}^{1+\frac{\gamma}{d}} \rho^*(z)^{1+\frac{\gamma}{d}} \mbox{ for } z\in \ro^d.
\end{equation}
By the theory of maximal functions, see e.g. \cite{StHar}, 
$$
\int_{\mR^d}\rho^*(z)^{1+\frac{\gamma}{d}} \, dz \le C \int_{\mR^d} \rho \psi (z)^{1+\frac{\gamma}{d}} \, dz, 
$$
it follows from \eqref{liOxln-p8} that 
\begin{equation}\label{liOxln-p9}
\int_{\ro^d}\int_0^\infty \min\{f_R^2(z),f_R(z)\} \frac{dR}{R^{d+\gamma+1}} \leq C_{LO} \int_{\ro^d} \rho\psi(z)^{1+\frac{\gamma}{d}} dz, 
\end{equation}
for some positive constant  $C_{LO}$ depending only on $d,\gamma$.

The conclusion now follows from \eqref{liOxln-p7} and \eqref{liOxln-p9}. The proof is complete. 
\qed

\providecommand{\bysame}{\leavevmode\hbox to3em{\hrulefill}\thinspace}
\providecommand{\MR}{\relax\ifhmode\unskip\space\fi MR }
\providecommand{\MRhref}[2]{%
  \href{http://www.ams.org/mathscinet-getitem?mr=#1}{#2}
}
\providecommand{\href}[2]{#2}

\end{document}